\documentclass[twoside,10pt, leqno]{amsart}
\usepackage{fixmath}
\usepackage{amsfonts}
\usepackage{amsmath}
\usepackage{amsrefs}
\usepackage{mathtools}
\usepackage{amsthm}
\usepackage{amssymb}
\usepackage{enumerate}
\usepackage{undertilde}
\begin{document}
\newtheorem{theorem}{Theorem}
\newtheorem{proposition}{Proposition}
\newtheorem{lemma}{Lemma}
\newtheorem{corollary}{Corollary}
\newtheorem{conjecture}{Conjecture}
\numberwithin{equation}{section}
\renewcommand{\thefootnote}{\fnsymbol{footnote}}
\newcommand{\dif}{\mathrm{d}}
\newcommand{\intz}{\mathbb{Z}}
\newcommand{\ratq}{\mathbb{Q}}
\newcommand{\natn}{\mathbb{N}}
\newcommand{\comc}{\mathbb{C}}
\newcommand{\rear}{\mathbb{R}}
\newcommand{\prip}{\mathbb{P}}
\newcommand{\uph}{\mathbb{H}}
\newcommand{\logl}{\mathcal{L}}

\title{\bf Gaps between Primes in Beatty Sequences}
\author{Roger C. Baker and Liangyi Zhao}
\date{\today}
\maketitle

\begin{abstract}
In this paper, we study the gaps between primes in Beatty sequences following the methods in the recent breakthrough of \cite{May}.
\end{abstract}

2010 Mathematics Subject Classficiation: 11B05, 11L20, 11N35

\section{Introduction}

Let $p_n$ denote the $n$-th prime and $t$ a natural number with $t \geq 2$.  It has long been conjectured that
\[ \liminf_{n \to \infty} \left( p_{n+t-1} - p_n \right) < \infty . \]
This was established recently for $t=2$ by Y. Zhang \cite{Zhang} and shortly after for all $t$ by J. Maynard \cite{May}.  Maynard showed that for $N > C(t)$, the interval $[N, 2N)$ contains a set $\mathcal{S}$ of $t$ primes of diameter
\[ D ( \mathcal{S} ) \ll t^3 \exp (4t) , \]
where
\[ D ( \mathcal{S} ) : = \max \{ n : n \in \mathcal{S} \} - \min \{ n : n \in \mathcal{S} \} . \]

In the present paper, we adapt Maynard's method to prove a similar result where $\mathcal{S}$ is contained in a prescribed set $\mathcal{A}$ (see Theorem~\ref{gentheo}).  We then work out applications (Theorems \ref{bea1} and \ref{bea2}) to a section of a Beatty sequence, so that
\[ \mathcal{A} = \{ [ \alpha m + \beta ] : m \geq 1 \} \cap [ N, 2N ) . \]
The number $\alpha$ is assumed to be irrational with $\alpha > 1$, while $\beta$ is a given real number.  We require an auxiliary result (Theorem~\ref{beabomvino}) for the estimation of errors of the form
\[ \sum_{\mathclap{\substack{N \leq n < N' \\ \gamma n \in I \bmod{1} \\ n \equiv a \bmod{q}}}} \Lambda(n) - \frac{(N-N') |I|}{\varphi(q)} , \]
where $I$ is an interval of length $|I| < 1$ and $\gamma = \alpha^{-1}$.  Theorem~\ref{beabomvino} is of ``Bombieri-Vinogradov type"; for completeness, we include a result of Barban-Davenport-Halberstam type for these errors (Theorem~\ref{beabardavhal}). \newline

We note that Chua, Park and Smith \cite{CPS} have already used Maynard's method to prove the existence of infinitely many sets of $k$ primes of diameter at most $C = C(\alpha, k)$ in a Beatty sequence $[\alpha n]$, where $\alpha$ is irrational and of finite type.  However, no explicit bound for $C$ is given. \newline

In this paragraph, we introduce some notations to be used throughout this paper.  We suppose that $t \in \natn$, $N \geq C(t)$ and write $\logl = \log N$, 
\[ D_0 = \frac{\log \logl}{\log \log \logl} . \]
Moreover, $(d,e)$ and $[d,e]$ stand for the great common divisor and the least common multiple of $d$ and $e$, respectively.  $\tau(q)$ and $\tau_k(q)$ are the usual divisor functions.  $\| x \|$ is the distance of between $x \in \rear$ and the nearest integer.  Set
\[ P(z) = \prod_{p < z} p \; \mbox{with} \; z \geq 2 \; \; \; \mbox{and} \; \; \; \psi(n,z) = \left\{ \begin{array}{cl} 1 & \mbox{if} \; (n, P(z)) = 1, \\ 0 & \mbox{otherwise}. \end{array} \right. \]
$X (E; n)$ stands for the indicator function of a set $E$ and $\prip$ for the set of primes.  Let $\varepsilon$ be a positive constant, sufficiently small in terms of $t$.  The implied constant ``$\ll$", when it appears, may depend on $\varepsilon$ and on $A$ (if $A$ appears in the statement of the result).  ``$F \asymp G$" means both $F \ll G$ and $G \ll F$ hold.  As usual, $e(y) = \exp(2\pi i y)$, and $o(1)$ indicates a quantity tending to 0 as $N$ tends to infinity.  Furthermore,
\[ \sum_{\chi \bmod{q}}  , \; \; \; \; \; \; \sideset{}{'} \sum_{\chi \bmod{q}} , \; \; \; \; \; \; \sideset{}{^{\star}} \sum_{\chi \bmod{q}} \]
denote, respectively, a sums over all Dirichlet characters modulo $q$, a sum over nonprincipal characters modulo $q$ and a sum restricted to primitive characters, other than $\chi = 1$, modulo $q$.  We write $\hat{\chi}$ for the primitive character that induces $\chi$.  A set $\mathcal{H} = \{ h_1 , \cdots, h_k \}$ of distinct non-negative integers is {\it admissible} if for every prime $p$, there is an integer $a_p$ such that $a_p \not\equiv h \pmod{p}$ for all $h \in \mathcal{H}$. \newline

In Sections 1 and 2, let $\theta$ be a positive constant.  Let $\mathcal{A}$ be a subset of $[N, 2N) \cap \natn$.  Suppose that $Y>0$ and $Y/q_0$ is an approximation to the cardinality of $\mathcal{A}$, $\# \mathcal{A}$.  Let $q_0$, $q_1$ be given natural numbers not exceeding $N$ with $(q_1, q_0P(D_0)) =1$ and $\varphi(q_1)=q_1 (1+o(1))$.  Suppose that $n \equiv a_0 \pmod{q_0}$ for all $n \in \mathcal{A}$ with $(a_0,q_0)=1$.  An admissible set $\mathcal{H}$ is given with
\[ h \equiv 0 \pmod{q_0} \; (h \in \mathcal{H}) \]
and
\begin{equation} \label{divcond}
p | h-h', \; \mbox{with} \; h, h' \in \mathcal{H}, \; h \neq h', p > D_0 \; \mbox{implies} \; p | q_0.
\end{equation}
We now state ``regularity conditions" on $\mathcal{A}$.
\begin{enumerate}[(I)]

\item We have
\begin{equation} \label{regcond1}
\sum_{\mathclap{\substack{q \leq N^{\theta} \\ (q, q_0q_1)=1}}} \mu^2(q) \tau_{3k}(q) \left| \sum_{n \equiv a_q \bmod{qq_0}} X ( \mathcal{A} ; n) - \frac{Y}{qq_0} \right| \ll \frac{Y}{q_0 \logl^{k+\varepsilon}}
\end{equation}
(any $a_q \equiv a_0 \pmod{q_0}$).

\item There are nonnegative functions $\varrho_1, \cdots , \varrho_s$ defined on $[N, 2N)$ (with $s$ a constant, $0 < a \leq s$) such that
\begin{equation} \label{regcond2}
X \left( \prip; n \right) \geq \varrho_1(n) + \cdots + \varrho_a(n) - \left( \varrho_{a+1} (n) + \cdots + \varrho_s (n) \right)
\end{equation}
for $n \in [N, 2N)$.  There are positive $Y_{g,m}$ ($g=1, \cdots, s$ and $m= 1, \cdots , k$) with
\[ Y_{g,m} = Y \left( b_{g,m} + o(1) \right) \logl^{-1} . \]
where the positive constants $b_{g,m}$ satisfy
\begin{equation} \label{bbound}
 b_{1,m} + \cdots + b_{a,m} - \left( b_{a+1,m} + \cdots + b_{s,m} \right) \geq b > 0 ,
 \end{equation}
for $m=1, \cdots, k$.  Moreover, for $m \leq k$, $g \leq s$ and any $a_q \equiv a_0 \pmod{q_0}$ with $(a_q,q)=1$ defined for $q \leq x^{\theta}$, $(q,q_0q_1)=1$, we have
\begin{equation} \label{regcond3}
\sum_{\substack{q \leq N^{\theta} \\ (q,q_0q_1)=1}} \mu^2(q) \tau_{3k} (q) \left| \sum_{n \equiv a_q \bmod{qq_0}} \varrho_g(n) X \left( \left( \mathcal{A} + h_m \right) \cap \mathcal{A} ; n \right) - \frac{Y_{g,m}}{\varphi(q_0q)} \right| \ll \frac{Y}{\varphi(q_0) \logl^{k+\varepsilon}} .
\end{equation}
Finally, $\varrho_g(n) =0$ unless $(n, P (N^{\theta/2})) = 1$.

\end{enumerate}

\begin{theorem} \label{gentheo}
Under the above hypotheses on $\mathcal{H}$ and $\mathcal{A}$, there is a set $\mathcal{S}$ of $t$ primes in $\mathcal{A}$ with diameter not exceeding $D(\mathcal{H})$, provided that $k \geq k_0 (t,b,\theta)$ ($k_0$ is defined at the end of this section).
\end{theorem}

In proving Theorem~\ref{bea1}, we shall take $s=a=1$, $q_0=q_1=1$, $\rho_1(n) = X(\prip; n)$.  A more complicated example with $s=5$, of the inequality \eqref{regcond2}, occurs in proving Theorem~\ref{bea2}, but again, $q_0=q_1=1$.  We shall consider elsewhere a result in which $q_0$, $q_1$ are large.  Maynard's Theorem 3.1 in \cite{May2} overlaps with our Theorem~\ref{gentheo}, but neither subsumes the other.

\begin{theorem} \label{bea1}
Let $\alpha >1$, $\gamma= \alpha^{-1}$ and $\beta \in \rear$.  Suppose that
\begin{equation} \label{alphacond}
\| \gamma r \| \gg r^{-3} 
\end{equation}
for all $r \in \natn$.  Then for any $N > c_1(t, \alpha, \beta)$, there is a set of $t$ primes of the form $[\alpha m + \beta]$ in $[N, 2N)$ having diameter
\[ < C_2 \alpha (\log \alpha + t) \exp (8t) , \]
where $C_2$ is an absolute constant.
\end{theorem}

\begin{theorem} \label{bea2}
Let $\alpha$ be irrational with $\alpha >1$ and $\beta \in \rear$.  Let $r \geq C_3 (\alpha, \beta)$ and
\[ \left| \frac{1}{\alpha} - \frac{b}{r} \right| < \frac{1}{r^2} , \; b \in \natn , \; (b,r) =1 . \]
Let $N = r^2$.  There is a set of $t$ primes of the form $[ \alpha n + \beta]$ in $[N, 2N)$ having diameter
\[ < C_4 \alpha \left( \log \alpha + t \right) \exp ( 7.743 t) , \]
where $C_4$ is an absolute constant.
\end{theorem}

Theorem~\ref{bea2} improves Theorem~\ref{bea1} in that $\alpha$ can be any irrational number in $(1,\infty)$ and $7.743 < 8$, but we lose the arbitrary placement of $N$. \newline

Turning our attention to our theorem of Bombieri-Vinogradov type, we write
\[ E (N, N', \gamma, q, a ) = \sup_I \left| \sum_{\substack{ N \leq n < N' \\ \gamma n \in I \bmod{1} \\ n \equiv a \bmod{q}}} \Lambda(n) - \frac{(N'-N)|I|}{\varphi(q)} \right| . \]
Here, $I$ runs over intervals of length $|I| < 1$.

\begin{theorem} \label{beabomvino}
Let $A>0$, $\gamma$ be a real number and $b/r$ a rational approximation to $\gamma$,
\begin{equation} \label{gammaratapprox}
\left| \gamma - \frac{b}{r} \right| \leq \frac{1}{r N^{3/4}} , \; N^{\varepsilon} \leq r \leq N^{3/4}, \; (b,r)=1 .
\end{equation}
Then for $N < N' \leq 2N$ and any $A>0$, we have
\begin{equation} \label{beabomvinoeq}
\sum_{q \leq \min (r , N^{1/4}) N^{-\varepsilon}} \max_{(a,q)=1} E (N, N', \gamma, q, a) \ll N \logl^{-A} .
\end{equation}
\end{theorem}

Our Barban-Davenport-Halberstam type result is the following.

\begin{theorem} \label{beabardavhal}
Let $A>0$ and $\gamma$ be an irrational number.  Suppose that for each $\eta > 0$ and sufficiently large $r \in \natn$, we have
\begin{equation} \label{gammatyperes}
\| \gamma r \| > \exp (- r^{\eta} ) .
\end{equation}
Let $N \logl^{-A} \leq R \leq N$.  Then for $N < N' \leq 2N$,
\begin{equation} \label{beabardavhaleq}
\sum_{q \leq R} \sum_{\substack{a=1 \\ (a,q)=1}}^q E(N, N', \gamma, q, a)^2 \ll NR \logl ( \log \logl )^2 .
\end{equation}
\end{theorem}

There are weaker results overlapping with Theorems~\ref{beabomvino} and \ref{beabardavhal}    by W. D. Banks and I. E. Shparlinski \cite{BaSh}. \newline

Let $\gamma$ be irrational, $\eta > 0$ and suppose that
\[ \| \gamma r \| \leq \exp \left( - r^{\eta} \right) \]
for infinitely many $r\in \natn$.  Then \eqref{beabardavhaleq} fails (so Theorem~\ref{beabardavhal} is optimal in this sense).  To see this, take $N = \exp (r^{\eta/2})$, $N'=2N$, $R = N \logl^{-8/\eta}$.  We have, for some $u \in \intz$,
\[ \left| \gamma n - \frac{un}{r} \right| \leq 2N r^{-1} \exp ( - r^{\eta} ) < \frac{1}{4r} , \; (n \leq 2N)  . \]
From this, we infer that
\[ \gamma n \not\in \left( \frac{1}{4r}, \frac{3}{4r} \right) \pmod{1} , \; (n \leq 2N) .\]
So
\[ E (N, 2N, \gamma, q,a)^2 \geq \frac{N^2}{4r^2\varphi(q)} , \; ( q \leq R, (a,q)=1) . \]
Therefore,
\[ \sum_{q \leq R} \sum_{\substack{a=1 \\ (a,q)=1}}^q E(N, 2N, \gamma, q, a)^2 \geq \frac{N^2}{4r^2} \sum_{q \leq R} \frac{1}{\varphi(q)} > \frac{N^2}{r^2} = NR \logl^{4/\eta} . \]

We now turn to the definition of $k_0(t,b,\theta)$.  For a smooth function $F$ supported on
\[ \mathcal{R}_k = \left\{ (x_1, \cdots, x_k) \in [0,1]^k : \sum_{i=1}^k x_i \leq 1 \right\} , \]
set
\[ I_k(F) = \int\limits_0^1 \cdots \int\limits_0^1 F(t_1, \cdots, t_k)^2 \dif t_1 \cdots \dif t_k , \]
and
\[ J^{(m)}_k(F) =\int\limits_0^1 \cdots \int\limits_0^1\left( \int\limits_0^1 F(t_1, \cdots, t_k) \dif t_m \right)^2 \dif t_1 \cdots \dif t_{m-1} \dif t_{m+1} \cdots \dif t_k  \]
for $m=1, \cdots, k$.  Let
\[ M_k = \sup_F \frac{\displaystyle{\sum_{m=1}^k J^{(m)}_k (F)}}{I_k(F)} , \]
where the $\sup$ is taken over all functions $F$ specified above and subject to the conditions $I_k(F) \neq 0$ and $J^{(m)}_k(F) \neq 0$ for each $m$.  Sharpening a result of Maynard \cite{May}, D. H. J. Polymath \cite{polym} gives the lower bound
\begin{equation} \label{mklowerbound}
M_k \geq \log k + O(1) .
\end{equation}
Now let $k_0 (t,b, \theta)$ be the least integer $k$ for which
\begin{equation} \label{k0defcond}
M_k > \frac{2t-2}{b\theta} .
\end{equation}

\section{Deduction of Theorem~\ref{gentheo} from Two Propositions}

We first write down some lemmas that we shall need later.

\begin{lemma} \label{GGPYlem}
Let $\kappa$, $A_1$, $A_2$, $L > 0$.  Suppose that $\gamma$ is a multiplicative function satisfying
\[ 0 \leq \frac{\gamma(p)}{p} \leq 1-A_1 \]
for all prime $p$ and
\[ -L \leq \sum_{w \leq p \leq z} \frac{\gamma(p) \log p}{p} - \kappa \log \frac{z}{w} \leq A_2 \]
for any $w$ and $z$ with $2 \leq w \leq z$.  Let $g$ be the totally multiplicative function defined by
\[ g(p) = \frac{\gamma(p)}{p-\gamma(p)} . \]
Suppose that $G: [0,1] \to \rear$ is a piecewise differentiable function with
\[ |G(y)| + |G'(y)| \leq B \]
for $0 \leq y \leq 1$ and
\begin{equation} \label{Sdef}
 S = \prod_p \left( 1- \frac{\gamma(p)}{p} \right)^{-1} \left( 1- \frac{1}{p} \right)^{\kappa} .
\end{equation}
Then for $z >1$, we have
\[ \sum_{d < z} \mu(d)^2 g(d) G \left( \frac{\log d}{\log z} \right) = \frac{S (\log z)^{\kappa}}{\Gamma(\kappa)} \int_0^1 t^{\kappa-1} G(t) \dif t + O \left( SLB (\log z)^{\kappa-1} \right)  . \]
The implied constant above depends on $A_1$, $A_2$, $\kappa$, but is independent of $L$.
\end{lemma}

\begin{proof}
This is \cite[Lemma 4]{GGPY}.
\end{proof}

Throughout this section, we assume that the hypotheses of Theorem~\ref{gentheo} hold.  Moreover, we write
\[ W_1 = \prod_{p \leq D_0 \; \mbox{or} \; \;  \; p | q_0q_1} p , \; \; \; W_2 = \prod_{\substack{p \leq D_0 \\ p \nmid q_0 }} p , \; \; \; R = N^{\theta/2-\varepsilon} . \]
Recalling the definition of admissible set, we pick a natural number $\nu_0$ with
\[ (\nu_0 + h_m, W_2) = 1 \; \; \; \; (m=1, \cdots, k) . \]

\begin{lemma} \label{multlem}
Suppose that $\gamma(p) = 1 + O(p^{-1})$ if $p \nmid W_1$ and $\gamma(p) = 0$ if $p | W_1$.  Let $\kappa =1$ and $S$ as defined in \eqref{Sdef}.  We have
\[ S = \frac{\varphi(W_1)}{W_1} \left( 1 + O(D_0^{-1}) \right) . \]
\end{lemma}

\begin{proof}
We have
\[ S = \prod_{p | W_1} \left( 1 - \frac{1}{p} \right) \prod_{p \nmid W_1} \left( 1 - \frac{1}{p} + O \left( \frac{1}{p^2} \right) \right)^{-1} \left( 1- \frac{1}{p} \right) = \frac{\varphi(W_1)}{W_1} \prod_{\substack{p > D_0 \\  p \nmid q_0q_1}} \left( 1+ O(p^{-2}) \right) , \]
from which the statement of the lemma can be readily obtained.
\end{proof}

\begin{lemma} \label{T1T2lem}
Let $H>1$,
\[ T_1 = \sum_{\substack{d \leq R \\ (d,W_1)=1}} \frac{\mu^2(d)}{d} \sum_{a | d} \frac{4^{\omega(a)}}{a}  \; \; \mbox{and} \; \; T_2 = \sum_{H < d \leq R} \frac{\mu^2(d)}{d^2} \sum_{a|d} a^{-1/2} .\]
Then, we have
\begin{equation} \label{T1bound}
T_1 \ll \frac{\varphi(W_1)}{W_1} \logl
\end{equation}
and
\begin{equation} \label{T2bound}
T_2 \ll H^{-1}  . 
\end{equation}
\end{lemma}

\begin{proof}
Let $\gamma(p)=0$ if $p | W_1$ and
\[ \gamma(p) = \frac{p^2+4p}{p^2+p+4} \]
if $p \nmid W_1$.  Then $g(p)$, as defined in the statement of Lemma~\ref{GGPYlem}, is
\[ g(p) = \frac{1}{p} + \frac{4}{p^2} \]
if $p \nmid W_1$.  Therefore, if $d$ is square-free and $(d,W_1)=1$,
\[ \frac{1}{d} \sum_{a|d} \frac{4^{\omega(a)}}{a} = \frac{1}{d} \prod_{p|d} \left( 1 + \frac{4}{p} \right) =g(d) . \]
Otherwise, if $(d,W_1) \neq 1$, then $g(d)=0$.  Using Lemma~\ref{GGPYlem} with $G(y)=1$ and Lemma~\ref{multlem}, we have
\[ T_1 = \sum_{d \leq R} \mu^2(d) g(d) G \left( \frac{ \log d}{\log R} \right) = \frac{\varphi(W_1)}{W_1} \left( 1 + O (D_0^{-1} ) \right) \log R + O \left( \frac{\varphi(W_1)}{W_1} L \right) , \]
where we can take
\[ L = \sum_{p | W_1} \frac{\log p}{p} \ll \log D_0 + \log \omega(q_0) \ll \log \logl . \]
Combining everything, we get \eqref{T1bound}.  \newline

To prove \eqref{T2bound}, we interchange the summations and get
\[ T_2 \leq \sum_{a \leq R} a^{-5/2} \sum_{Ha^{-1} < k \leq Ra^{-1}} k^{-2} \ll \sum_{a \leq R} a^{-3/2} H^{-1} \ll H^{-1} , \]
completing the proof of the lemma.
\end{proof}

\begin{lemma} \label{multfunction}
Let $f_0$, $f_1$ be multiplicative functions with $f_0(p)=f_1(p)+1$.  Then for squarefree $d$, $e$,
\[ \frac{1}{f_0 ([d,e])} = \frac{1}{f_0(d) f_0(e)} \sum_{k | d,e} f_1 (k) . \]
\end{lemma}

\begin{proof}
We have
\[  \frac{1}{f_0(d) f_0(e)} \sum_{k | d,e} f_1 (k) =  \frac{1}{f_0(d) f_0(e)} \prod_{p | (d,e)} \left( 1+ f_1 (p) \right) =  \frac{1}{f_0(d) f_0(e)} \prod_{p | (d,e)} f_0 (p) = \prod_{p | [d,e]} (f_0 (p))^{-1} . \] 
The lemma follows from this.
\end{proof}

We now prove two propositions that readily yield Theorem~\ref{gentheo} when combined.  To state them, we define weights $y_{\mathbf{r}}$ and $\lambda_{\mathbf{r}}$ for tuples
\[ \mathbf{r} = ( r_1, \cdots, r_k ) \in \natn^k \]
having the properties
\begin{equation} \label{tupprop}
\left( \prod_{i=1}^k r_i, W_1 \right) =1, \; \mu^2 \left( \prod_{i=1}^k r_i \right) = 1.
\end{equation}
We set $y_{\mathbf{r}} = \lambda_{\mathbf{r}} = 0$ for all other tuples.  Let $F$ be a smooth function with $|F| \leq 1$ and the properties given at the end of Section 1.  Let
\begin{equation} \label{yrdef}
y_{\mathbf{r}} = F \left( \frac{\log r_1}{\log R} , \cdots , \frac{\log r_k}{\log R} \right) ,
\end{equation}
and
\begin{equation} \label{lamddef}
\lambda_{\mathbf{d}} = \prod_{i=1}^k \mu(d_i) d_i \sum_{\substack{\mathbf{r} \\ d_i | r_i \; \forall i}} \frac{y_{\mathbf{r}}}{\prod_{i=1}^k \varphi(r_i)} .
\end{equation}

We have
\begin{equation} \label{lambbound}
 \lambda_{\mathbf{r}} \ll \logl^k 
\end{equation}
(see (5.9) of \cite{May}).  For $n \equiv \nu_0 \pmod{W_2}$, let
\[ w_n = \left( \sum_{d_i | n+h_i \; \forall i} \lambda_{\mathbf{d}} \right)^2 , \]
and $w_n=0$ for all other natural numbers $n$. \newline

\begin{proposition} \label{S1prop}
Let
\[ S_1 = \sum_{N \leq n < 2N} w_n X ( \mathcal{A} ; n) . \]
Then
\[ S_1 = \frac{(1+o(1)) \varphi(W_1)^k Y (\log R)^k I_k (F)}{q_0 W_1^k W_2} . \]
\end{proposition}

\begin{proposition} \label{S2prop}
Let 
\[ S_2 (g, m) = \sum_{\substack{N \leq n < 2N \\ n \in \mathcal{A} \cap (\mathcal{A} - h_m)}} w_n \varrho_g(n+h_m) . \]
Then for $ 1 \leq g \leq s$ and $1 \leq m \leq k$,
\[ S_2 (g,m) = \frac{b_{g,m} (1+o(1)) \varphi(W_1)^{k+1} Y (\log R)^{k+1} J^{(m)}_k (F)}{\varphi(q_0) \varphi(W_2) W_1^{k+1} \logl} . \]
\end{proposition}

Before proving the above propositions, we shall deduce Theorem~\ref{gentheo} from them.
\begin{proof}[Proof of Theorem~\ref{gentheo}]
Let
\[ Z = \frac{Y\varphi(W_1)^k}{q_0W_1^kW_2} (\log R)^k \]
and
\[ S(N) = \sum_{\substack{N \leq n < 2N \\ n \in \mathcal{A}}} w_n \left( \sum_{m=1}^k X \left( \prip \cap \mathcal{A}; n + h_m \right) - (t-1) \right) .  \]
Since $w_n \geq 0$, \eqref{regcond2} gives that
\[ S(N) \geq \sum_{m=1}^k \left( \sum_{g=1}^a S_2(g,m) - \sum_{g=a+1}^s S_2 (g,m) \right) - (t-1)S_1 . \]
Using Propositions \ref{S1prop} and \ref{S2prop}, the right-hand side of the above is
\[ \left(1+o(1) \right) Z \left( \sum_{m=1}^k \left( \sum_{g=1}^a b_{g,m} - \sum_{g=a+1}^s b_{g,m} \right) J^{(m)}_k(F) \left( \frac{\theta}{2} - \varepsilon \right) - (t-1) I_k(F) \right) . \]
Here we have used
\[ \frac{\varphi(q_0) \varphi(q_1) \varphi(W_2)}{q_0q_1 W_2} \frac{W_1}{\varphi(W_1)} = 1 \; \; \; \mbox{and} \; \; \; \frac{\varphi(q_1)}{q_1} = 1 + o(1) . \]

Therefore, using \eqref{bbound}, we get
\[ S(N) \geq \left(1+o(1) \right) Z \left( b\sum_{m=1}^k J^{(m)}_k(F) \left( \frac{\theta}{2} - \varepsilon \right) - (t-1) I_k(F) \right)  > 0 , \]
for a suitable choice of $F$.  The positivity of the above expression is a consequence of \eqref{k0defcond}.  Therefore, there must be at least one $n \in \mathcal{A}$ for which
\[ \sum_{m=1}^k X \left( \prip \cap \mathcal{A} ; n+ h_m \right) > t-1 . \]
For this $n$, there is a set of $t$ primes $n+h_{m_1}, \cdots , \; n+ h_{m_t}$ in $\mathcal{A}$.
\end{proof}

\section{Proof of Propositions \ref{S1prop} and \ref{S2prop}}

This section is devoted to the proofs of the two propositions.

\begin{proof}[Proof of Proposition~\ref{S1prop}]
We first show that
\begin{equation} \label{S11st}
S_1 = \frac{Y}{q_0W_2} \sum_{\mathbf{r}} \frac{y_{\mathbf{r}}^2}{\prod_{i=1}^k \varphi(r_i)} + O \left( \frac{Y \varphi(W_1)^k \logl^k}{q_0W_2W_1^kD_0} \right).
\end{equation}
From the definition of $w_n$, we get
\begin{equation} \label{S1wn}
S_1 = \sum_{\mathbf{d}, \; \mathbf{e}} \lambda_{\mathbf{d}} \lambda_{\mathbf{e}} \sum_{\substack{N \leq n < 2N \\ n \equiv \nu_0 \bmod{W_2} \\ [d_i, e_i ] | n+h_i \; \forall i}} X \left( \mathcal{A}; n \right). 
\end{equation}

Recall that $n \equiv a_0 \pmod{q_0}$ for all $n \in \mathcal{A}$.  The inner sum of the above takes the form
\[ \sum_{\substack{N \leq n < 2N \\ n \equiv a_q \bmod{qq_0}}} X \left( \mathcal{A}; n \right), \]
where
\[ q = W_2 \prod_{i=1}^k [d_i, e_i] , \]
provided that $W_2$, $[d_1, e_1]$, $\cdots$, $[d_k, e_k]$ are pairwise coprime.  The latter restriction reduces to
\begin{equation} \label{coprimcond}
(d_i, e_j) = 1
\end{equation}
for all $i \neq j$, and we exhibit this condition on the summation by writing
\[ \sideset{}{'} \sum_{\mathbf{d}, \; \mathbf{e}} . \]

Outside of $ \sum_{\mathbf{d}, \; \mathbf{e}}^{'}$, the inner sum is empty.  To see this, suppose that $p | d_i$, $p|e_j$ with $i \neq j$, then the conditions
\[ [d_i, e_i] | n+h_i , \; \mbox{and} \; [d_j, e_j] | n+h_j \]
imply that $p | h_i-h_j$.  This means that either $p \leq D_0$ or $p| q_0$, both contrary to $p|d_i$. \newline

Counting the number of times a given $q$ can arise, we get
\begin{equation} \label{countq}
S_1 - \frac{Y}{q_0W_2} \sideset{}{'} \sum_{\mathbf{d}, \; \mathbf{e}} \frac{\lambda_{\mathbf{d}} \lambda_{\mathbf{e}}}{ \prod_{i=1}^k [d_i, e_i]} \ll \left( \max_{\mathbf{d}} \left| \lambda_{\mathbf{d}} \right| \right)^2 \sum_{\substack{q \leq R^2 W_2 \\ (q,q_0)=1}} \mu^2(q) \tau_{3k} (q) \left| \sum_{n \equiv a_q \bmod{qq_0}} X \left( \mathcal{A}; n \right) - \frac{Y}{qq_0} \right| .
\end{equation}

Since $R^2 W_2 \leq N^{\theta}$, we can appeal to \eqref{regcond1} and \eqref{lambbound} to majorize the right-hand side of \eqref{countq} by
\[ \ll \frac{Y}{q_0} \logl^{2k-(k+\varepsilon)} \ll \frac{\varphi(W_1)^k Y \logl^k}{q_0 W_2 W_1^k D_0} . \]

Applying Lemma~\ref{multfunction} with $f_1 = \varphi$, we see that
\[ S_1 = \frac{Y}{q_0W_2} \sum_{\mathbf{u}} \prod_{i=1}^k \varphi(u_i) \sideset{}{'} \sum_{\substack{\mathbf{d}, \; \mathbf{e} \\ u_i | d_i, e_i \; \forall i}} \frac{\lambda_{\mathbf{d}} \lambda_{\mathbf{e}}}{ \prod_{i=1}^k d_i e_i}  + O \left( \frac{\varphi(W_1)^k Y \logl^k}{q_0 W_2 W_1^k D_0} \right) . \]

Now we follow \cite{May} verbatim to transform this equation into
\begin{equation} \label{maytrans}
S_1 = \frac{Y}{q_0 W_2} \sum_{\mathbf{u}} \prod_{i=1}^k \varphi(u_i) \sideset{}{^*} \sum_{s_{1,2}, \cdots, s_{k,k-1}} \prod_{\substack{1 \leq i, j \leq k \\ i \neq j}} \mu \left( s_{i,j} \right) \sum_{\substack{\mathbf{d}, \; \mathbf{e} \\ u_i | d_i, e_i \; \forall i \\ s_{i,j} | d_i, e_j \; \forall i \neq j}} \frac{\lambda_{\mathbf{d}} \lambda_{\mathbf{e}}}{ \prod_{i=1}^k d_i e_i}  + O \left( \frac{\varphi(W_1)^k Y \logl^k}{q_0 W_2 W_1^k D_0} \right) . 
\end{equation}
Here $\sum^{*}$ indicates that $(s_{i,j}, u_iu_j )=1$ and $(s_{i,j} , s_{i, c}) = 1 = (s_{i,j} , s_{d,j})$, for $c \neq j$, $d \neq i$.  Now define
\begin{equation} \label{ajbjdef}
a_j = u_j \prod_{i \neq j} s_{j,i}, \; \; b_j = u_j \prod_{i \neq j} s_{i,j} .
\end{equation}
As in \cite{May}, we recast \eqref{maytrans} as
\begin{equation} \label{maytrans2}
S_1 = \frac{Y}{q_0 W_2} \sum_{\mathbf{u}} \prod_{i=1}^k \frac{\mu(u_i)^2}{\varphi(u_i)} \sideset{}{^*} \sum_{s_{1,2}, \cdots, s_{k,k-1}} \prod_{\substack{1 \leq i, j \leq k \\ i \neq j}} \mu \left( s_{i,j} \right) \sum_{\substack{\mathbf{d}, \; \mathbf{e} \\ u_i | d_i, e_i \; \forall i \\ s_{ij} | d_i, e_j \; \forall i \neq j}} \frac{\mu(s_{i,j})}{\varphi(s_{i,j})^2}  y_{\mathbf{a}} y_{\mathbf{b}}  + O \left( \frac{\varphi(W_1)^k Y \logl^k}{q_0 W_2 W_1^k D_0} \right) . 
\end{equation}

For the non-zero terms on the right-hand side of \eqref{maytrans2}, either $s_{i,j}=1$ or $s_{i,j} > D_0$.  The terms of the latter kind (for given $i, j, i \neq j$) contribute
\begin{equation} \label{maytrans2bound}
\ll \frac{Y}{q_0W_2} \left( \sum_{\substack{ u < R \\ (u, W_1)=1}} \frac{\mu(u)^2}{\varphi(u)} \right)^k \left( \sum_{s_{i,j} > D_0} \frac{\mu(s_{i,j})^2}{\varphi(s_{i,j})^2} \right) \left( \sum_{s \geq 1} \frac{\mu(s)}{\varphi(s)^2} \right)^{k^2-k-1} = \frac{Y}{q_0W_2} U_1 U_2 U_3 ,
\end{equation}
say.  Clearly, $U_3 \ll 1$.  Now if $u$ is squarefree, we have
\[ \frac{1}{\varphi(u)} = \frac{1}{u} \prod_{p |u} \left( 1 - \frac{1}{p} \right)^{-1} \ll \frac{1}{u} \sum_{a|u} \frac{1}{a}  \]
and
\[ \frac{1}{\varphi(u)^2} \ll \frac{1}{u^2} \prod_{p |u} \left( 1 + \frac{2}{p} \right) = \frac{1}{u^2} \sum_{a|u} \frac{2^{\omega(a)}}{a} \ll \frac{1}{u^2} \sum_{a|u} a^{-1/2}. \]
So \eqref{T1bound} and \eqref{T2bound} give, respectively,
\[ U_1 \ll \left( \frac{\varphi(W_1)}{W_1} \logl \right)^k \; \; \;  \mbox{and} \; \; \; U_2 \ll \frac{1}{D_0} . \]
Hence, the right-hand side of \eqref{maytrans2bound} is
\[ \ll \frac{\varphi(W_1)^k Y \logl^k}{q_0 W_2 W_1^k D_0} \]
and we have \eqref{S11st}. \newline

Now, we shall deduce Proposition~\ref{S1prop} from \eqref{S11st}.  Mindful of \eqref{lamddef}, we have
\[ S_1 = \frac{Y}{q_0W_2} \sum_{\substack{ \mathbf{u} \\ (u_l, u_j) =1 \; \forall l \neq j \\ (u_l , W_1) =1 \; \forall l}} \prod_{i=1}^k \frac{\mu(u_i)^2}{\varphi(u_i)} F \left( \frac{\log u_1}{\log R} , \cdots , \frac{\log u_k}{\log R} \right)^2 +  O \left( \frac{\varphi(W_1)^k Y \logl^k}{q_0 W_2 W_1^k D_0} \right) . \]
Note that the common prime factors of two integers both coprime to $W_1$ are strictly greater than $D_0$.  Thus, we may drop the condition $(u_l, u_j)=1$ in the above expression at the cost of an error of size
\[ \ll \frac{Y}{q_0W_2} \sum_{p > D_0} \sum_{\substack{u_1 \cdots u_k < R \\ p | u_l, u_j \\ (u_l, W_1)=1 \; \forall l}} \prod_{i=1}^k \frac{\mu(u_i)^2}{\varphi(u_i)} \ll \frac{Y}{q_0W_2} \sum_{p > D_0} \frac{1}{(p-1)^2} \left( \sum_{\substack{ u < R \\ (u,W_1)=1}} \frac{\mu(u)^2}{\varphi(u)} \right)^k \ll \frac{\varphi(W_1)^k Y \logl^k}{q_0 W_2 W_1^k D_0} , \]
by virtue of \eqref{T1bound}. \newline

It remains to evaluate the sum
\begin{equation} \label{S1mt}
\sum_{\substack{\mathbf{u} \\ (u_l, W_1)=1 \; \forall l}} \prod_{i=1}^k \frac{\mu(u_i)^2}{\varphi(u_i)} F \left( \frac{\log u_1}{\log R} , \cdots , \frac{\log u_k}{\log R} \right)^2.
\end{equation}
This requires applying Lemma~\ref{GGPYlem} $k$ times with
\[ \gamma (p) = \left\{ \begin{array}{cl} 0 & p | W_1, \\ 1 & p \nmid W_1 . \end{array} \right. \]
We take $A_1$ and $A_2$ to be suitable constants and
\[ L \ll 1 + \sum_{p | W_1} \frac{\log p}{p} \ll \log \logl  \]
as noted earlier.  In the $j$-th application, we replace the summation over $u_j$ by the integeral over $[0,1]$.  Ultimately, we express the sum in \eqref{S1mt} in the form
\[ \frac{\varphi(W_1)^k}{W_1^k} \left( \log R \right)^k I_k (F) + O \left( \frac{\varphi(W_1) (\log \logl) \logl^{k-1}}{W_1^k} \right) \]
and Proposition~\ref{S1prop} follows at once.
\end{proof}

We shall need the following lemma in the proof of Proposition~\ref{S2prop}.

\begin{lemma} \label{rm1ylem}
Let $1 \leq m \leq k$ and suppose that $r_m=1$.  Let
\[ y_{\mathbf{r}}^{(m)} = \prod_{i=1}^k \mu(r_i) g(r_i) \sum_{\substack{ \mathbf{d} \\ r_i | d_i \forall i \\ d_m=1}} \frac{\lambda_{\mathbf{d}}}{\prod_{i=1}^k \varphi(d_i)} . \]
Then
\[ y_{\mathbf{r}}^{(m)} = \sum_{a_m} \frac{y_{r_1, \cdots, r_{m-1}, a_m, r_{m+1}, \cdots , r_k}}{\varphi(a_m)} + O \left( \frac{\varphi(W_1) \logl}{W_1 D_0} \right). \]
\end{lemma}

\begin{proof}
Following \cite{May} verbatim, we have
\begin{equation} \label{lem5start}
 y_{\mathbf{r}}^{(m)} = \prod_{i=1}^k \mu(r_i) g(r_i) \sum_{\substack{\mathbf{a} \\ r_i | a_i \forall i}} \frac{y_{\mathbf{a}}}{\prod_{i=1}^k \varphi(a_i)} \prod_{i\neq m} \frac{\mu(a_i)r_i}{\varphi(a_i)} .
 \end{equation}
Fix $j$, $1 \leq j \leq k$.  In \eqref{lem5start}, the nonzero terms will have either $a_j = r_j$ or $a_j > D_0r_j$.  The contribution from the terms with $a_j \neq r_j$ is
\begin{equation} \label{ajneqrj}
\ll \prod_{i=1}^k g(r_i) r_i \left( \sum_{\substack{a_j > D_0r_j \\ r_j | a_j}} \frac{\mu(a_j)^2}{\varphi(a_j)^2} \right) \left( \sum_{\substack{a_m < R \\ (a_m, W_1)=1}} \frac{\mu(a_m)^2}{\varphi(a_m)} \right) \prod_{\substack{ 1 \leq i \leq k \\ i \neq j, m}} \sum_{r_i | a_i} \frac{\mu(a_i)^2}{\varphi(a_i)^2} .
\end{equation}
Now, as before, from \eqref{T1bound} and \eqref{T2bound},
\[ \sum_{\substack{a_j > D_0r_j \\ r_j | a_j}} \frac{\mu(a_j)^2}{\varphi(a_j)^2} \ll \frac{1}{D_0 \varphi(r_j)^2} , \; \sum_{\substack{a_m < R \\ (a_m, W_1)=1}} \frac{\mu(a_m)^2}{\varphi(a_m)} \ll \frac{\varphi(W_1)}{W_1} \logl \]
and
\[ \sum_{r_i | a_i} \frac{\mu(a_i)^2}{\varphi(a_i)^2} \leq \frac{\mu(r_i)^2}{\varphi(r_i)^2} \sum_k \frac{\mu(k)}{\varphi(k)^2} \ll \frac{1}{\varphi(r_i)^2}, \]
majorizing \eqref{ajneqrj} by
\[ \ll \prod_{i=1}^k \frac{g(r_i)r_i}{\varphi(r_i)^2} \frac{\varphi(W_1)}{W_1 D_0} \logl \ll \frac{\varphi(W_1) \logl}{W_1 D_0} . \]
Hence \eqref{lem5start} becomes
\[  y_{\mathbf{r}}^{(m)} = \prod_{i=1}^k \frac{g(r_i)r_i}{\varphi(r_i)^2} \sum_{a_m} \frac{y_{r_1, \cdots, r_{m-1}, a_m, r_{m+1}, \cdots , r_k}}{\varphi(a_m)} + O \left( \frac{\varphi(W_1) \logl}{W_1 D_0} \right) , \]
and the proof is completed by applying Lemma~\ref{multlem}.
\end{proof}

Now we proceed to the proof of Proposition~\ref{S2prop}.

\begin{proof}[Proof of Proposition~\ref{S2prop}]
Let
\[ y_{\max}^{(m)} = \max_{\mathbf{r}} \left| y_{\mathbf{r}}^{(m)} \right|, \]
where $y_{\mathbf{r}}^{(m)}$ is defined in Lemma~\ref{rm1ylem}.  We shall first show that
\begin{equation} \label{firstthingprop2}
S_2 (g,m) = \frac{Y_{g,m}}{\varphi(q_0) \varphi(W_2)} \sum_{\mathbf{u}} \frac{\left( y_{\mathbf{u}}^{(m)}\right)^2}{\prod_{i=1}^k g(u_i)} + O \left( \frac{Y \logl^{k-2} \varphi^{k-1}(W_1) \left( y_{\max}^{(m)} \right)^2}{\varphi(q_0) \varphi(W_2) W_1^{k-1} D_0} + \frac{Y \logl^{k-\varepsilon}}{\varphi(q_0)} \right) .
\end{equation}
From the definition of $w_n$, we have
\begin{equation} \label{S2withWn}
S_2(g,m) = \sum_{\mathbf{d}, \mathbf{e}} \lambda_{\mathbf{d}} \lambda_{\mathbf{e}} \sum_{\substack{n \in \mathcal{A} \cap ( \mathcal{A} - h_m ) \\ N \leq n < 2N , \; n \equiv \nu_0 \bmod{W_2} \\ [d_i, e_i] | n+h_i \; \forall i}} \varrho_g(n+h_m) .
\end{equation}
As in the proof of Proposition~\ref{S1prop}, $\sum_{\mathbf{d}, \mathbf{e}}$ reduces to $\sum_{\mathbf{d}, \mathbf{e}}^{'}$.  Let $n' = n+h_m$.  Since $n+h_m \equiv a_0 \pmod{q_0}$ for $n \in \mathcal{A}$, the inner sum of \eqref{S2withWn} reduces to
\[ T( \mathbf{d}, \mathbf{e} ) : = \sum_{\substack{ n' \equiv \nu_0 + h_m \bmod{W_2} \\ n' \equiv a_0 \bmod{q_0} \\ n' \equiv h_m-h_i \bmod{[d_i,e_i]} \forall i}} X \left( \mathcal{A} \cap ( \mathcal{A} + h_m ) , n' \right) \varrho_g(n'). \]
Recall that $\varrho_g(n') = 0$ if $n'$ is divisible by a prime divisor of $[d_i,e_i]$.  Since one condition of the summation is $[d_m , e_m] | n'$, we have $T(\mathbf{d}, \mathbf{e}) = 0$ unless $d_m=e_m=1$.  When $d_m=e_m=1$,
\[ T(\mathbf{d}, \mathbf{e}) = \sum_{n \equiv a_q \bmod{qq_0}} X \left( \mathcal{A} \cap (\mathcal{A}+h_m) , n \right) \varrho_g(n) . \]
Here we have
\[ q = W_2 \prod_{i=1}^k [d_i, e_i], \; (a_q, q)=1, \; a_q \equiv a_0 \pmod{q_0} . \]
For $(a_q,q)=1$, we need $(h_m-h_i, [d_i,e_i])=1$ whenever $m \neq i$, which was noted earlier. \newline

Arguing as in the proof of Proposition~\ref{S1prop}, \eqref{regcond3} now gives
\[ S_2(g,m) = \frac{Y_{g,m}}{\varphi(q_0) \varphi(W_2)} \sideset{}{'} \sum_{\substack{\mathbf{d}, \mathbf{e} \\ d_m=e_m=1}} \frac{\lambda_{\mathbf{d}} \lambda_{\mathbf{e}}}{\prod_{i=1}^k \varphi([d_i,e_i])} + O \left( \frac{Y \logl^{k-\varepsilon}}{\varphi(q_0)} \right). \]

With $a_j$ and $b_j$ as in \eqref{ajbjdef}, we follow \cite{May} to obtain
\begin{equation} \label{S2afterMay}
S_2(g,m) = \frac{Y_{g,m}}{\varphi(q_0)\varphi(W_2)} \sum_{\mathbf{u}} \prod_{i=1}^k \frac{\mu(u_i)^2}{g(u_i)} \sideset{}{^*} \sum_{s_{1,2}, \cdots, s_{k,k-1}} \prod_{\substack{1 \leq i,j \leq k \\ i \neq j}} \frac{\mu(s_{i,j}}{g^2(s_{i,j})} y_{\mathbf{a}}^{(m)} y_{\mathbf{b}}^{(m)} + O \left( \frac{Y \logl^{k-\varepsilon}}{\varphi(q_0)} \right).
\end{equation}
Here $q$ is the totally multiplicative function with $g(p) = p-2$ for all $p$ and we have used Lemma~\ref{multfunction} with $f_1=g$. \newline

The contribution to the sum in \eqref{S2afterMay} from $s_{i,j} \neq 1$ (for given $i,j$) is
\begin{equation} \label{not1contrib}
\begin{split}
\ll \frac{Y \left( y_{\max}^{(m)} \right)^2}{\varphi(q_0) \varphi(W_2) \logl} & \left( \sum_{\substack{ u < R \\ (u, W_1)=1}} \frac{\mu(u)^2}{g(u)} \right)^{k-1} \left( \sum_s \frac{\mu(s)^2}{g(s)^2} \right)^{k(k-1)-1} \left( \sum_{s_{i,j}>D_0} \frac{\mu(s_{i,j})^2}{g(s_{i,j})^2} \right) \\ & = \frac{Y \left( y_{\max}^{(m)} \right)^2}{\varphi(q_0) \varphi(W_2) \logl} V_1 V_2 V_3,
\end{split}
\end{equation}
say.  Clearly, $V_2 \ll 1$.  Using \eqref{T1bound} while mindful of the estimate
\[ \frac{1}{g(s)} \ll \frac{1}{s} \sum_{a|s} \frac{2^{\omega(a)}}{a} \]
yields that
\[ V_1 \ll \left( \frac{\varphi(W_1)}{W_1} \logl \right)^{k-1} . \]
From \eqref{T2bound} and the observation that, for $s$ squarefree,
\[ \frac{1}{g^2(s)} \ll \frac{1}{s^2} \sum_{a|s} \frac{4^{\omega(a)}}{a} \ll \frac{1}{s^2} \sum_{a|s} a^{-1/2}  , \]
we get that
\[ V_3 \ll D_0^{-1} . \]
Note the bound in \eqref{not1contrib} is
\[ \ll  \frac{Y \left( y_{\max}^{(m)} \right)^2 \logl^{k-2}}{\varphi(q_0) \varphi(W_2)} \left( \frac{\varphi(W_1)}{W_1} \right)^{k-1} \frac{1}{D_0} , \]
and we have established \eqref{firstthingprop2}. \newline

Now we use Lemma~\ref{rm1ylem} in \eqref{firstthingprop2}, recalling \eqref{yrdef}.  When $r_m=1$,
\begin{equation} \label{yrm=1}
\begin{split}
y_{\mathbf{r}}^{(m)} = \sum_{\left( u, W_1 \prod_{i=1}^k r_i \right)=1} \frac{\mu(u)^2}{\varphi(u)} F &\left( \frac{\log r_1}{\log R} , \cdots , \frac{\log r_{m-1}}{\log R} , \frac{\log u}{\log R} , \frac{\log r_{m+1}}{\log R}, \cdots , \frac{\log r_k}{\log R} \right) \\
& + O \left( \frac{\varphi(W_1) \logl}{ W_1 D_0} \right) .
\end{split}
\end{equation}
From this, we find that
\[ y_{\max}^{(m)} \ll \frac{\varphi(W_1)}{W_1} \logl . \]
We shall apply Lemma~\ref{GGPYlem} to \eqref{yrm=1} with $\kappa=1$,
\[ \gamma(p) = \left\{ \begin{array}{cl} 1, & p \nmid W_1 \prod_{i=1}^k r_i \\ 0 , & \mbox{otherwise}, \end{array} \right. \]
$A_1$, $A_2$ suitably chosen and
\[ L \ll \log \logl  \]
(similar to the proof of \eqref{T1bound}).  Define
\[ F_{\mathbf{r}}^{(m)} = \int\limits_0^1 F \left( \frac{\log r_1}{\log R} , \cdots , \frac{\log r_{m-1}}{\log R} , t_m , \frac{\log r_{m+1}}{\log R}, \cdots , \frac{\log r_k}{\log R} \right) \dif t_m . \]
We obtain that
\[ y_{\mathbf{r}}^{(m)} = \log R \frac{\varphi(W_1)}{W_1} \left( \prod_{i=1}^k \frac{\varphi(r_i)}{r_i} \right) F_{\mathbf{r}}^{(m)} + O \left( \frac{\varphi(W_1) \logl}{W_1 D_0} \right) . \]
Inserted into \eqref{firstthingprop2}, the above produces a main term
\begin{equation} \label{usefirstthing}
\frac{(\log R)^2 Y_{g,m} \varphi(W_1)^2}{\varphi(q_0) \varphi(W_2) W_1^2} \sum_{\substack{ \mathbf{r} \\ (r_i, W_1)=1 \forall i \\ (r_i, r_j)=1 \forall i \neq j \\ r_m=1}} \prod_{i=1}^k \frac{\varphi(r_i) \mu(r_i)^2}{g(r_i)r_i^2} \left( F_{\mathbf{r}}^{(m)} \right)^2
\end{equation}
and an error term of size
\[ \ll \frac{Y_{g,m}}{\varphi(q_0)\varphi(W_2)} \sum_{\substack{ \mathbf{r} \\ r_m=1}} \frac{\varphi(W_1)^2 \logl^2}{W_1^2 D_0 \prod_{i=1}^k g(r_i)} \ll  \frac{Y\varphi(W_1)^2 \logl^2}{\varphi(q_0)\varphi(W_2) W_1^2 D_0} \left( \sum_{\substack{r < R \\ (r,W_1)=1}} \frac{1}{g(r)} \right)^{k-1} \ll \frac{Y\varphi(W_1)^{k+1} \logl^k}{\varphi(q_0)\varphi(W_2) W_1^{k+1} D_0} . \]
Recall that $Y_{g,m} \ll Y \logl^{-1}$.  Now we remove the condition $(r_i,r_j)=1$ from \eqref{usefirstthing}.  As before, this introduces an error of size
\[ \ll \frac{\logl^2 Y \varphi(W_1)^2}{\varphi(q_0) \varphi(W_2) W_1^2} \left( \sum_{p > D_0} \frac{\varphi(p)^2}{g(p)^2 p^2} \right) \left( \sum_{\substack{r <R \\ (r,W_1)=1}} \frac{\mu(r)^2 \varphi(r)}{g(r) r} \right)^{k-1} \ll \frac{Y\logl^k \varphi(W_1)^{k+1}}{\varphi(q_0)\varphi(W_2) W_1^{k+1} D_0}  \]
by an application of Lemma~\ref{T1T2lem}.  Combining all our results, we get
\[ S_2(g,m) = \frac{(\log R)^2 Y_{g,m} \varphi(W_1)^2}{\varphi(q_0) \varphi(W_2) W_1^2} \sum_{\substack{ \mathbf{r} \\ (r_i, W_1)=1 \forall i \\ r_m=1}} \prod_{i=1}^k \frac{\varphi(r_i)^2 \mu(r_i)^2}{g(r_i) r_i^2} \left( F_{\mathbf{r}}^{(m)} \right)^2 + O \left( \frac{Y\varphi(W_1)^{k+1} \logl^k}{\varphi(q_0) \varphi(W_2) W_1^{k+1} D_0} \right) . \]
The last sum is evaluated by applying Lemma~\ref{GGPYlem} to each summation variable in turn, taking
\[ \gamma(p) = \left\{ \begin{array}{cl} \frac{p^3-2p^2+p}{p^3-p^2-2p+1} , & p \nmid W_1 \\ 0 , & p | W_1 \end{array} \right. \]
to produce the right value of $\gamma(p)/(p-\gamma(p))$.  Of course
\[ S = \frac{\varphi(W_1)}{W_1} \left( 1 + O (D_0^{-1}) \right) \]
by Lemma~\ref{multlem}, while $L \ll \log \logl$.  Our final conclusion is that
\[ S_2(g,m) = \frac{(\log R)^{k+1} Y_{g,m} \varphi(W_1)^{k+1} J_k^{(m)}}{\varphi(q_0) \varphi(W_2) W_1^{k+1}} \left( 1 + o (1) \right) \]
completing the proof.
\end{proof}

\section{Further Lemmas}

Let $\gamma= \alpha^{-1}$.  As noted in \cite{BaSh}, the set of $[ \alpha m + \beta ]$ in $[N, 2N)$ may be written as
\[ \{ n \in [N, 2N) : \gamma n \in ( \gamma \beta - \gamma, \beta \gamma] \pmod{1} \} . \]

\begin{lemma} \label{intervallem}
Let $I=(a,b)$ be an interval of length $l$ with $0 < l < 1$ and let $h$ be a natural number satisfying
\[ 0 < -h \gamma < 2 \varepsilon \pmod{1} , \]
where $2 \varepsilon < l$.  Let
\[ \mathcal{A} = \{ n \in [N, 2N) : \gamma n \in I \pmod{1} \} . \]
Then
\[ \mathcal{A} \cap (\mathcal{A}+h) = \{ n \in [N+h, 2N) : \gamma n \in J \pmod{1} \} \]
where $J$ is an interval of length $l'$ with
\[ l - 2 \varepsilon < l' < l .\]
\end{lemma}

\begin{proof}
Let $t \equiv -h \gamma \pmod{1}$, $0 < t < 2 \varepsilon$.  Clearly $\mathcal{A} \cap ( \mathcal{A}+h)$ consists of the integers in $[N+h, 2N)$ for which
\[ \gamma n \in (a,b) \pmod{1}, \; \gamma n + t \in (a,b) \pmod{1} .\]
The lemma follows with $J=(a,b-t)$.
\end{proof}

\begin{lemma} \label{Baklem}
Let $I$ be an interval of length $l$, $0 < l < 1$.  Let $x_1$, $\cdots$, $x_N$ be real.  Then
\begin{enumerate}[(i)]
\item \label{Baklem1} There exists $z$ such that
\[ \# \left\{ j \leq N : x_j \in z + I \pmod{1} \right\} \geq N l. \]
\item \label{Baklem2} We have (for $a_j \geq 0$, $j=1, \cdots, N$ and $L \geq 1$)
\[ \sum_{\mathclap{\substack{j=1 \\ x_j \in I \bmod{1}}}}^N a_j - l \sum_{j=1}^N a_j \ll L^{-1} \sum_{j=1}^N a_j + \sum_{h=1}^L h^{-1} \left| \sum_{j=1}^N a_j e(h x_j) \right| . \]
\end{enumerate}
\end{lemma}

\begin{proof}
We leave (1) as an exercise;  (2) is a slight variant of \cite[Theorem 2.1]{Bak}.
\end{proof}

\begin{lemma} \label{splitdiag}
Let $1 \leq Q \leq N$ and $F$ a nonnegative function defined on Dirichlet characters.  Then for some $Q_1$, $1 \leq Q_1 \leq Q$,
\[ \sum_{q \leq Q} \ \sideset{} {'} \sum_{\chi \bmod{q}} F ( \hat{\chi} ) \ll \frac{\logl Q}{Q_1} \sum_{Q_1 \leq q_1 < 2Q_1}  \ \sideset{} {^{\star}} \sum_{\psi \bmod{q_1}} F(\psi) . \]
\end{lemma}

\begin{proof}
We recall that $\hat{\chi}$ is the primitive character that induces $\chi$, so that $F( \hat{\chi} )$ may be quite different from $F (\chi)$. \newline

The left-hand side of the claimed inequality is
\[ \sum_{q_1 \leq Q} \ \sideset{} {^{\star}} \sum_{\psi \bmod{q_1}} F( \psi) \sum_{\substack{\chi \bmod{q} \\ q \leq Q, \; q_1 | q \\ \psi \; \mathrm{induces} \; \chi}} 1 \leq \sum_{q_1 \leq Q} \ \sideset{} {^{\star}} \sum_{\psi \bmod{q_1}} F ( \psi ) \frac{Q}{q_1} .  \]

The lemma follows on applying a splitting-up argument to $q_1$.
\end{proof}

\begin{lemma} \label{BaBalem}
Let $f(j)$ ($j \geq 1$) be a periodic function with period $q$,
\[ S(f, n) = \sum_{j=1}^n f(j) e \left( - \frac{nj}{q} \right) ,\]
$F>0$, and $R \geq 1$.  Let $H(y)$ be a real function with $H'(y)$ monotonic and
\[ |H'(y))| \leq Fy^{-1} \]
for $R \leq y \leq 2R$.  Then for $J= [R,R']$ with $R < R' \leq 2R$,
\[ \sum_{m \in J} f(m) H(m) - q^{-1} \sum_{1 \leq |n| \leq 2FqR^{-1}} S(f,n) \int\limits_J e \left( \frac{ny}{q} + H(y) \right) \dif y \ll \frac{R|S(f,0)|}{qF} + \sum_{|n| \in J'} \frac{|S(f,n)|}{n} , \]
where
\[ J' = [ \min \{2FqR^{-1}, q/2 \} , \max \{ 2FqR^{-1} , q \} + q ] . \]
\end{lemma}

\begin{proof}
This is \cite[Theorem 8]{BaBa}.
\end{proof}

For a finite sequence $\{ a_k : K \leq k < K' \}$, set
\[ \| a \|_2 = \left( \sum_{K \leq k < K'} |a_k|^2 \right)^{1/2} . \]

\begin{lemma}
Let $R \geq 1$, $M \geq 1$, $H \geq 1$.  Let $\beta$ be real and
\begin{equation} \label{betadiriapprox}
 \left| \beta - \frac{u_1}{r_1} \right| \leq \frac{H}{r_1^2}
\end{equation}
where $r_1 \geq H$ and $(u_1, r_1)=1$.  Then for $M_1 \in \natn$,
\begin{equation} \label{spacingbound}
\sum_{m=M_1+1}^{M_1+M} \min \left( R, \frac{1}{\| m \beta \|} \right) \ll \left( \frac{HM}{r_1} + 1 \right) \left( R + r_1 \log r_1 \right) .
\end{equation}
If $M < r_1$ and
\[ M \left| \beta - \frac{u_1}{r_1} \right| \leq \frac{1}{2r_1} , \]
then
\begin{equation} \label{spacingbound2}
\sum_{m=1}^M \frac{1}{\| m \beta \|} \ll r_1 \log 2 r_1 .
\end{equation}
\end{lemma}

\begin{proof}
For \eqref{spacingbound}, it suffices to show that a block of $[r_1/H]$ consecutive $m$'s contribute
\[ \ll R + \sum_{l=1}^{r_1} \frac{r_1}{l} . \]
Writing $m = m_0 +j$, $1 \leq j \leq [r_1/H]$,
\[ \left| (m_0+j)\beta - m_0 \beta - \frac{ju_1}{r_1} \right| \leq \frac{jH}{r_1^2} \leq \frac{1}{r_1} , \]
so there are $O(1)$ values of $j$ for which the bound
\[ \| (m_0 + j ) \beta \| \geq \frac{1}{2} \left\| m_0 \beta + \frac{ju_1}{r_1} \right\| \]
fails.  Our block estimate follows immediately. \newline

The argument for \eqref{spacingbound2} is similar.  In this case, 
\[ \left| m\beta - \frac{mu_1}{r_1} \right| \leq \frac{1}{2r_1} , \]
if $ 1 \leq m \leq M$.  Therefore, the left-hand side of \eqref{spacingbound2} can be estimated by $\sum_{l=1}^{r_1} r_1/l$.
\end{proof}

\begin{lemma} \label{doubleaddmultsum}
Let $N < N' \leq 2N$, $MK \asymp N$, $N \geq K \geq M \geq 1$.  Suppose that
\begin{equation} \label{modgammacond}
\left| \gamma - \frac{u}{r} \right| \leq \frac{H}{r^2} , \; (u,r)=1, \; H \leq r \leq N .
\end{equation}
Let $(a_m)_{M \leq m < 2M}$, $(b_k)_{K \leq k < 2K}$ be two sequences of complex numbers.  Then
\begin{equation} \label{Ssumdef}
S : = \sum_{Q \leq q < 2Q} \sum_{\chi \bmod{q}} \left| \mathop{\sum_{M \leq m < 2M} \sum_{K \leq k < 2K}}_{N \leq mk < N'}  a_m b_k \chi (mk) e (\gamma mk) \right|
\end{equation}
satisfies the bound
\[ S \ll \| a \|_2 \| b \|_2 \logl^{3/2} D^{1/2} \left( Q^2 M^{1/2} + \frac{Q^{3/2} H^{1/2} N^{1/2}}{r^{1/2}} + Q^{3/2} H^{1/2} K^{1/2} + Q^{3/2} r^{1/2} \right) , \]
where 
\[ D = \max_{n < N} \# \{ q \in [Q, 2Q) : n =lq \}. \]
\end{lemma}

\begin{proof}
Let $S'$ be the sum obtained from $S$ by removing the condition $N \leq mk <N'$.  It suffices to prove the same bound, with $\logl^{1/2}$ in place of $\logl^{3/2}$, for $S'$, since the condition can be restored at the cost of a factor of $\logl$.  See \cite[Section 3.2]{GH}. \newline

We have
\[ S' \leq \sum_{Q \leq q < 2Q} \sum_{\chi \bmod{q}} \sum_{M \leq m < 2M} |a_m| \left| \sum_{K \leq k < 2K} b_k \chi(k) e(\gamma mk) \right| = \sum_q S_q, \]
say.  We may also assume that $b_k = 0$ if $(k,q)>1$.  By Cauchy's inequality, and with summations subject to the obvious restrictions on $m$, $k_1$ and $k_2$,
\[ S_q^2 \leq \varphi(q) \| a \|_2^2 \sum_{\chi \bmod{q}} \sum_m \sum_{k_1} \sum_{k_2} b_{k_1} \overline{b}_{k_2} \chi (k_1) \overline{\chi} (k_2) e ( \gamma m (k_1-k_2) ) . \]
Bringing the sum over $\chi$ inside we see that the right-hand side of the above is
\[ \varphi(q)^2 \| a \|_2^2 \sum_{\substack{k_1, k_2 \\ k_1 \equiv k_2 \bmod{q}}} b_{k_1} \overline{b}_{k_2} \sum_m e ( \gamma m (k_1-k_2) ) \leq \varphi(q)^2 \| a \|_2^2 \sum_{k_1} |b_{k_1}|^2 \sum_{k_1 \equiv k_2 \bmod{q}} \left|  \sum_m e ( \gamma m (k_1-k_2) ) \right| \]
upon using the parallelogram rule
\[ \left| b_{k_1} b_{k_2} \right| \leq \frac{1}{2} \left( \left| b_{k_1} \right|^2 + \left| b_{k_2} \right|^2 \right) . \]
Now summing the geometric sum over $m$ and then summing over $q$, we see that
\begin{equation} \label{aveSqsq}
\sum_{Q \leq q < 2Q} S_q^2 \ll  Q^3 \| a \|_2^2 \| b \|_2^2  M + Q^2 \| a \|_2^2 \| b \|_2^2 \sum_{Q \leq q < 2Q} \sum_{1 \leq l < K/q} \min \left( M, \frac{1}{\| \gamma l q \|} \right) .
\end{equation}
Now we combine the variables $l$ and $q$ and then apply \eqref{spacingbound}, leading to
\begin{equation*}
\begin{split}
 \sum_{Q \leq q < 2Q} S_q^2 & \ll  Q^3 \| a \|_2^2 \| b \|_2^2  M + Q^2 \| a \|_2^2 \| b \|_2^2 D \left( \frac{HK}{r} + 1 \right) \left( M + r \log r \right) \\
 & \ll \| a \|_2^2 \| b \|_2^2  \left( Q^3 M + \logl Q^2 D \left( \frac{HN}{r}+ HK +M +r \right) \right). 
 \end{split}
 \end{equation*}
The desired bound for $S'$ follows by another application of Cauchy's inequality.
\end{proof}

\begin{lemma} \label{Ssumestnew}
Under the hypotheses of Lemma~\ref{doubleaddmultsum}, suppose that $4MQ <N$, $b_k =1$ for $K \leq k < 2K$ and $|a_m| \leq 1$ for $M \leq m < 2M$.  Define $D$ as in Lemma~\ref{doubleaddmultsum}.  Then
\begin{enumerate}[(i)]
\item \label{doubleaddmultsum2} We have
\[ S \ll Q^{3/2} \logl D \left( \frac{QMH}{r} + 1 \right) \left( \frac{K}{Q} + r \right) . \]
\item \label{doubleaddmultsum3} If $4MQ < r$ and
\[ 4MQ \left| \gamma - \frac{u}{r} \right| \leq \frac{1}{2r} , \]
then
\[ S \ll \logl D Q^{3/2} r . \]
\end{enumerate}
\end{lemma}
\begin{proof}
Let $I_m$ (here and after) denote a subinterval of $[N/m, N'/m)$.  We have
\[ S \leq QS^* + S^{**} , \]
where, for a suitably chosen nonprincipal $\chi_q \pmod{q}$,
\[ S^* = \sum_{Q \leq q < 2Q} \sum_{M \leq m < 2M} \left| \sum_{k \in I_m} \chi_q(k) e (\gamma mk) \right| \]
and
\[ S^{**} = \sum_{Q \leq q < 2Q} \sum_{M \leq m < 2M} \left| \sum_{k \in I_m} \chi_0 (k) e (\gamma m k) \right| . \]
To prove part \eqref{doubleaddmultsum2}, it suffices to show that
\[ S^* \ll Q^{1/2} \logl D \left( \frac{QMH}{r} + 1 \right)  \left( \frac{K}{Q} + r \right) \]
and
\[ S^{**} \ll Q \logl D \left( \frac{QMH}{r} + 1 \right) \left( \frac{K}{Q} + 1 \right) . \]
We give the proof for $S^*$; the proof for $S^{**}$ is similar. \newline

Given $q$ and $m$, Lemma~\ref{BaBalem} together with, using the notation from Lemma~\ref{BaBalem},
\[ \left| S (\chi, q) \right| \leq \sqrt{q} \]
(see Chapter 9 of \cite{HD}) gives
\begin{equation*}
\begin{split}
\sum_{k \in I_m} \chi_q(k) e (\gamma mk) - \frac{1}{q} \sum_{1 \leq |n| \ll Mq} S(\chi_q, n) & \int\limits_{I_m} e \left( \left( \frac{n}{q} + \gamma m \right) y \right) \dif y \\
& \ll q^{-1/2} M^{-1} + q^{1/2} \sum_{1 \leq n \ll Mq} n^{-1} \ll q^{1/2} \logl .
\end{split}
\end{equation*}
Therefore
\[ \sum_{k \in I_m} \chi_q(k) e (\gamma mk) \ll q^{1/2} \logl + q^{-1/2} \sum_{1 \leq |n| \ll Mq} \min \left( K, \frac{1}{\left| \gamma m - \frac{n}{q} \right|} \right). \]
Summing over $m$ and $q$,
\[ S^* \ll MQ^{3/2} \logl + Q^{1/2} \sum_{Q \leq q < 2Q} \sum_{M \leq m < 2M} \sum_{1 \leq |n| \ll Mq}  \min \left( \frac{K}{Q}, \frac{1}{\left| \gamma mq - n \right|} \right). \]
The contribution to the right-hand side of the above from $n$'s with $|n-\gamma mq| > 1/2$ is
\[ \ll MQ^{3/2} \logl . \]
Now combining the variables $m$, $q$,
\begin{equation} \label{S*est}
S^* \ll MQ^{3/2} \logl + Q^{1/2}D \sum_{MQ \leq m' < 4MQ} \min \left( \frac{K}{Q}, \frac{1}{\| \gamma m' \|} \right) .
\end{equation}
We can now deduce the desired bound for $S^*$ by applying \eqref{spacingbound}. \newline

Now for part \eqref{doubleaddmultsum3}, we note that \eqref{spacingbound2} is applicable to the reciprocal sum in \eqref{S*est} with $4MQ$ and $\gamma$ in place of $M$ and $\beta$.  Hence
\[ S^* \ll MQ^{3/2} \logl + Q^{1/2} D r \log 2r \ll D \logl Q^{1/2} r  \]
since $4MQ < r$.  Similarly $S^{**} \ll D \logl Q r$, and part \eqref{doubleaddmultsum3} follows.
\end{proof}

\begin{lemma} \label{SsumtoW}
Suppose that
\[ \left| \gamma - \frac{u}{r} \right| \leq \frac{\logl^{A+1}}{r^2} \]
with $(u,r)=1$ and that $r^2 \leq N \leq r^2 \logl^{2A+2}$.  Then
\begin{enumerate}[(i)]
\item \label{Ssumesta} For $Q < N^{2/7-\varepsilon}$, $N^{4/7} \ll K \ll N^{5/7}$ and any $a_m$, $b_k$ with $|a_m| \leq \tau(m)^B$, $|b_k| \leq \tau(k)^B$, where $B$ is an absolute constant, the sum $S$ in \eqref{Ssumdef} satisfies the bound
\begin{equation} \label{Sest1}
S \ll QN^{1-\varepsilon/4}.
\end{equation}
\item \label{Ssumestb} For $Q \leq N^{2/7-\varepsilon}$, $M \ll N^{4/7}$ and $b_k =1$ for $K \leq k < 2K$, $|a_m| \leq 1$ for $M \leq m < 2M$, the sum $S$ in \eqref{Ssumdef} satisfies \eqref{Sest1}.
\end{enumerate}
\end{lemma}

\begin{proof}
In order to prove \eqref{Ssumesta}, we use Lemma~\ref{doubleaddmultsum}.  As $D \ll N^{\varepsilon/15}$,
\[ SQ^{-1} N^{-1+\varepsilon/4} \ll Q^{-1} N^{-1/2+ \varepsilon/3} \left( Q^2 N^{3/14} + Q^{3/2} N^{5/14} \right) \ll N^{-1/2+\varepsilon/2} \left( QN^{3/14} + Q^{1/2} N^{5/14} \right) \ll 1. \]

To prove \eqref{Ssumestb}, we break the situation into two cases.  If $K < N^{1-\varepsilon}$, then by \eqref{doubleaddmultsum2} of Lemma~\ref{Ssumestnew}
\[ SQ^{-1} N^{-1+\varepsilon/4} \ll Q^{1/2} N^{-1+\varepsilon/2} \left( N^{1/2} + MQ + \frac{N^{1-\varepsilon}}{Q} \right) \ll N^{1/7-1/2+\varepsilon} + N^{3/7+4/7-1-\varepsilon} + N^{-\varepsilon/2} \ll 1. \]
If $K \geq N^{1-\varepsilon}$, then $M \ll N^{\varepsilon}$ and \eqref{doubleaddmultsum3} of Lemma~\ref{Ssumestnew} is applicable since
\[ 4MQ \left| \gamma - \frac{u}{r} \right| \ll N^{-1+2/7+\varepsilon} . \]
Hence
\[ SQ^{-1} N^{-1+\varepsilon/4} \ll Q^{1/2} N^{-1/2+\varepsilon} \ll 1 , \]
giving the desired majorant.
\end{proof}

\begin{lemma} \label{hdid}
Let $f$ be an arbitrary complex function on $[N,2N)$.  Let $N < N' \leq 2N$.  The sum
\[ S = \sum_{N \leq n < N'} \Lambda(n) f(n) \]
can be decomposed into $O ( \logl^2 )$ sums of the form
\[ \sum_{M<m \leq 2M} a_m \sum_{\substack{K \leq k < 2K \\ N \leq mk < N'}} f(mk) \; \; \; \;\mbox{or} \; \; \; \; \int\limits_N^{N'} \sum_{M \leq m < 2M} a_m \sum_{\substack{k \geq w \\ K \leq k < 2K \\ N \leq mk < N'}} f(mk) \frac{\dif w}{w} \]
with $M \leq N^{1/4}$ and $|a_m| \leq 1$, together with $O (\logl)$ sums of the form
\[  \sum_{M<m \leq 2M} a_m \sum_{\substack{K \leq k < 2K \\ N \leq mk < N'}} b_k f(mk) \]
with $N^{1/2} \leq K \ll N^{3/4}$ and $\| a \|_2 \| b \|_2 \ll N^{1/2} \logl^2$.
\end{lemma}

\begin{proof}
This follows from the arguments in \cite[Chapter 24]{HD} by taking $U=V=N^{1/4}$.
\end{proof}

We record a special case of \cite[Lemma 14]{BaWe}.  For more background on the ``Harman sieve", see \cite{GH}.

\begin{lemma} \label{BaWelem}
Let $W(n)$ be a complex function with support in $(N, 2N] \cap \intz$, $|W(n)| \leq N^{1/\varepsilon}$.  For $r \in \natn$, $z \geq 2$, let
\begin{equation}
 S^*(r,z) = \sum_{(n, P(z))=1} W(rn) .
 \end{equation}
 Suppose that for some constant $c>0$, $0 \leq d \leq 1/2$, and for some $Y > 0$, we have, for any coefficients $a_m$, $b_k$ with $|a_m| \leq 1$, $|b_k| \leq \tau(k)$,
 \begin{equation} \label{type1sumest}
 \sum_{m \leq 2N^c} a_m \sum_k W(mk) \ll Y,
 \end{equation}
 and
 \begin{equation} \label{type2sumest}
 \sum_{N^c \leq m \leq 2N^{c+d}} a_m \sum_k b_k W(mk) \ll Y.
 \end{equation}
 Let $u_r$ ($r \leq N^c$) be complex numbers with $|u_r| \leq 1$ and $u_r=0$ for $\left( r, P(N^{\varepsilon}) \right) >1$.  Then
 \[ \sum_{r \leq (2N)^c} u_r S^* \left( r, (2N)^d \right) \ll Y \logl^3 . \]
\end{lemma}

The following application of Lemma~\ref{BaWelem} will be used in the proof of Theorem~\ref{bea2}.  We take
\begin{equation} \label{Wdef}
W(n) = \sum_{Q \leq q < 2Q} \sum_{\chi \bmod{q}} \eta_{\chi} \chi(n) e( \gamma n) 
\end{equation}
for $N \leq n < N'$; otherwise, $W(n)=0$.  Here $\eta_{\chi}$ is arbitrary with $| \eta_{\chi}| \leq 1$.

\begin{lemma} \label{S*withweights}
Suppose that
\[ \left| \gamma - \frac{u}{r} \right| \leq \frac{\logl^{A+1}}{r^2} , \; (u,r)=1, \; N=r^2 , \; 1 \leq Q \leq N^{2/7-\varepsilon} . \]
Define $S^*(r,z)$ as above with $W$ defined in \eqref{Wdef}.  Then
\[ \sum_{r \leq (2N)^{4/7}} u_r S^* \left( r, (2N)^{1/7} \right) \ll N \logl^{-A} \]
for every $A>0$, provided that $|u_r| \leq 1$, $u_r=0$ for $\left( r, P(N^{\varepsilon}) \right) > 1$.
\end{lemma}

\begin{proof}
We need to verify \eqref{type1sumest} and \eqref{type2sumest} with $c=4/7$, $d=1/7$ and $Y=N \logl^{-A-3}$.  This is an application of Lemma~\ref{SsumtoW}.
\end{proof}

We now introduce some subsets of $\rear^j$ needed in the proof of Theorem~\ref{bea2}.  Write $E_j$ for the set of $j$-tuples $\mathbold{\alpha}_j = ( \alpha_1, \cdots, \alpha_j)$ satisfying
\[ \frac{1}{7} \leq \alpha_j < \alpha_{j-1} < \cdots < \alpha_1 \leq \frac{1}{2} \; \; \mbox{and} \; \; \alpha_1 + \alpha_2 + \cdots + \alpha_{j-1} + 2 \alpha_j \leq 1. \]
A tuple $\mathbold{\alpha}_j$ is said to be {\it good} if some subsum of $\alpha_1+ \cdots + \alpha_j$ is in $[2/7, 3/7] \cup [4/7, 5/7]$ and {\it bad} otherwise. \newline

We use the notation $p_j = (2N)^{\alpha_j}$.  For instance, the sum
\[ \sum_{\substack{p_1 p_2 n_3 =k \\ (2N)^{1/7} \leq p_2 < p_1 < (2N)^{1/2}}} \psi( n_3, p_2) \]
will be written as
\[ \sum_{\substack{p_1p_2n_3 = k \\ \mathbold{\alpha}_2 \in E_2}} \psi(n_3, p_2) . \]

\begin{lemma} \label{primesumest}
Let $\gamma$, $u/r$, $N$, $Q$ be as in Lemma~\ref{S*withweights} and $E$ be a subset of $E_j$ defined by a bounded number of inequalities of the form
\begin{equation} \label{ineqgenform}
c_1 \alpha_1 + \cdots + c_j \alpha_j < c_{j+1} \; (\mbox{or} \; \leq c_{j+1}).
\end{equation}
Suppose that all points in $E$ are good and that throughout $E$, $z_j$ is either the function $z_j = (2N)^{\alpha_j}$ or the constant $z_j = (2N)^{1/7}$.  Then for arbitrary $\eta_{\chi}$ with $| \eta_{\chi}| \leq 1$,
\[ \sum_{Q \leq q < 2Q} \sum_{\chi \bmod{q}} \eta_{\chi} \sum_{\substack{N \leq p_1 \cdots p_j n_{j+1} < N' \\ \mathbold{\alpha}_j \in E}} \chi(p_1 \cdots p_j n_{j+1}) e ( \gamma p_1 \cdots p_j n_{j+1}) \psi( n_{j+1} , z_j ) \ll N \logl^{-A} , \]
for every $A>0$.
\end{lemma}

\begin{proof}
This is a consequence of \eqref{Ssumesta} of Lemma~\ref{SsumtoW}.  On grouping a subset of the variables as a product $m=\prod_{i \in \mathcal{S}} p_i$, with $S \subset \{ 1, \cdots, j \}$, we obtain a sum $S$ of the form appearing in \eqref{Ssumesta} of Lemma~\ref{SsumtoW}, except that a bounded number of inequalities of the form \eqref{ineqgenform} are present.  These inequalities may be removed at the cost of a log power, by the mechanism noted earlier.  See page 184 of \cite{BaWe} for a few more details of a similar argument.  The lemma follows at once.
\end{proof}

\begin{lemma} \label{buchstabdecomp}
Let $D= \{ (\alpha_1, \alpha_2) \in E_2 : (\alpha_1, \alpha_2) \; \mbox{is bad}, \; \alpha_1 + 2 \alpha_2 > 5/7 \}$.  Then
\[ X \left( \prip ; n \right) - \sum_{\substack{p_1 p_2 n_3 = n \\ \mathbold{\alpha}_2 \in D}} \psi( n_3, p_2) = \varrho_1(n) + \varrho_2(n) + \varrho_3(n) - \varrho_4(n) - \varrho_5(n) . \]
Here
\[ \varrho_1(n) = \psi(n, (2N)^{1/7}) , \; \varrho_4 (n) = \sum_{\substack{p_1 n_2 = n \\ \mathbold{\alpha}_1 \in E_1}} \psi(n_2, (2N)^{1/7}) , \; \varrho_2(n) = \sum_{\substack{p_1p_2n_3=n \\ \mathbold{\alpha}_2 \in E_2 \setminus D}} \psi(n_3 , (2N)^{1/7}) , \]
\[ \varrho_5(n) = \sum_{\substack{p_1p_2p_3n_4=n \\ \mathbold{\alpha}_3 \in E_3 \\ (\alpha_1, \alpha_2) \in E_2 \setminus D}} \psi(n_4, (2N)^{1/7}) \; \; \; \mbox{and} \; \; \; \varrho_3(n) = \sum_{\substack{p_1p_2p_3p_4n_5=n \\ \mathbold{\alpha}_4 \in E_4 \\ (\alpha_1, \alpha_2) \in E_2 \setminus D}} \psi(n_5, p_4) . \]
\end{lemma}

\begin{proof}
We repeatedly use Buchstab's identity in the form
\[ \psi(m,z) = \psi(m,w) - \sum_{\substack{ph=m\\ w \leq p < z}} \psi(h,p) \; \; \; ( 2 \leq w < z) . \]
Thus
\begin{equation} \label{buchstab1}
\begin{split}
X ( \prip ; n) & = \psi(n, (2N)^{1/2}) = \psi(n, (2N)^{1/7}) - \sum_{\substack{(2N)^{1/7} \leq p_1 < (2N)^{1/2} \\ p_1n_2=n}} \psi(n_2, p_1) \\
& = \varrho_1(n) - \varrho_4(n) + \sum_{\substack{p_1p_2n_3 =n \\ \mathbold{\alpha}_2 \in E_2}} \psi(n_3,p_2) ,
\end{split}
\end{equation}
\[ X( \prip ; n ) - \sum_{\substack{p_1p_2n_3=n \\ \mathbold{\alpha}_2 \in D}} \psi(n_3 , p_2) = \varrho_1(n) - \varrho_4(n) + \sum_{\substack{p_1p_2n_3=n \\ \mathbold{\alpha}_2 \in E_2 \setminus D}} \psi(n_3 , p_2) . \]
Continuing the decomposition of the last sum,
\begin{equation} \label{buchstab2}
\sum_{\substack{p_1p_2n_3=n \\ \mathbold{\alpha}_2 \in E_2 \setminus D}} \psi(n_3 , p_2) = \sum_{\substack{p_1p_2n_3=n \\ \mathbold{\alpha}_2 \in E_2 \setminus D}} \psi( n_3 , (2N)^{1/7} ) - \sum_{\substack{p_1p_2p_3n_4=n \\ \mathbold{\alpha}_3 \in E_3 \\ (\alpha_1, \alpha_2) \in E_2 \setminus D}} \psi( n_4, (2N)^{1/7} ) + \sum_{\substack{p_1 p_2 p_3 p_4 n_5 = n \\ \mathbold{\alpha}_4 \in E_4 \\ (\alpha_1, \alpha_2) \in E_2 \setminus D}} \psi(n_5 , p_4) .
\end{equation}
Combining \eqref{buchstab1} and \eqref{buchstab2}, we complete the proof of the lemma.
\end{proof}

\begin{lemma} \label{primechardetect}
Let $r$, $u/r$, $N$ and $Q$ be as in Lemma~\ref{S*withweights} with $\varrho_1$, $\cdots$, $\varrho_5$ as in Lemma~\ref{buchstabdecomp}; we have
\[ \sum_{Q \leq q < 2Q} \sum_{\chi \bmod{q}} \eta_{\chi} \sum_{N \leq n < N
} \varrho_j (n) \chi(n) e(\gamma n) \ll QN \logl^{-A} \]
for arbitrary $\eta_{\chi}$ with $|\eta_{\chi}| \leq 1$ and any $A > 0$.
\end{lemma}

\begin{proof}
This follows from Lemmas \ref{S*withweights} and \ref{primesumest} for $j=1,2, 4, 5$ on noting that $\alpha_1 + \alpha_2 + \alpha_3 \leq \alpha_1 + 2 \alpha_2 \leq 5/7$ for $j=5$, so that either $\mathbold{\alpha}_3$ is good or $\alpha_1 + \alpha_2 + \alpha_3 < 4/7$ (similarly for $j=2$).  For $j=3$, we need to show that each $\mathbold{\alpha}_4$ counted is good.  Suppose that some $\mathbold{\alpha}_4$ is bad.  We have $\alpha_1 + \alpha_2 + \alpha_3 + 2 \alpha_4 \leq 1$.  Hence $\alpha_1 + \alpha_2 + \alpha_3 \leq 5/7$ from which we infer that $\alpha_1 + \alpha_2 + \alpha_3 < 4/7$.  Therefore, $\alpha_1 + \alpha_2 < 3/7$.  But we know that $\alpha_1 + \alpha_2 > 2/7$.  This makes $\mathbold{\alpha}_4$ good, a contradiction. 
\end{proof}

\section{Proof of Theorems \ref{beabomvino} and \ref{beabardavhal}}

\begin{proof}[Proof of Theorem \ref{beabomvino}]
With a suitable choice of $a_q$, $(a_q, q) =1$, we have
\begin{equation*}
\begin{split}
 \max_{(a,q)=1} E (N, N', \gamma, q, a) & \leq \sup_I \left| \sum_{\substack{N \leq n < N' \\ \gamma n \in I \bmod{1} \\ n \equiv a_q \bmod{q}}} \Lambda(n) - |I| \sum_{\substack{N \leq n < N' \\ n \equiv a_q \bmod{q}}} \Lambda(n) \right| + \left| \sum_{\substack{N \leq n < N' \\ n \equiv a_q \bmod{q}}} \Lambda(n) - \frac{N'-N}{\varphi(q)} \right| \\
 &= T_1(q) + T_2(q),
 \end{split}
 \end{equation*}
say.  In view of the Bombieri-Vinogradov theorem, we need only bound $\sum_q T_1(q)$, which is, applying Lemma~\ref{Baklem},
\[ \ll \sum_{q \leq N^{1/4-\varepsilon}} \logl^{-A-1} \sum_{\substack{N \leq n < N' \\ n \equiv a_q \bmod{q}}} \Lambda(n) + \sum_{q \leq \min (r, N^{1/4})N^{-\varepsilon}} \sum_{h \leq \logl^{A+1}} \frac{1}{h} \left| \sum_{\substack{N \leq n < N' \\ n \equiv a_q \bmod{q}}} \Lambda(n) e( \gamma nh) \right| . \]
Let $H= \logl^{A+1}$.  Mindful of the Brun-Titchmarsh inequality, it remains to show that for $1 \leq h \leq H$,
\[ \sum_{q \leq \min (N^{1/4}, r) N^{-\varepsilon}} \left| \sum_{\substack{N \leq n < N' \\ n \equiv a_q \bmod{q}}} \Lambda(n) e ( \gamma nh) \right| \ll N \logl^{-A-1} . \]
Reducing $hu/r$ into lowest terms, we need only show that
\[ \sum_{q \leq \min (N^{1/4}, r) N^{-\varepsilon/2}} \eta_q \sum_{\substack{N \leq n < N' \\ n \equiv a_q \bmod{q}}} \Lambda(n) e (\gamma n) \ll N \logl^{-A-1} \]
under the modified hypothesis \eqref{modgammacond} on $\gamma$ (with $H= \logl^{A+1}$), whenever $| \eta_q | \leq 1$. \newline

Using Lemma~\ref{hdid}, it suffices to show that
\begin{equation} \label{afterhdid}
\sum_{q \leq \min (N^{1/4}, r) N^{-\varepsilon/2}} \eta_q \mathop{\sum_{M \leq m < 2M} \sum_{K \leq k < 2K}}_{\substack{N \leq mk < N' \\ mk \equiv a_q \bmod{q}}} a_m b_k e (\gamma mk)  \ll N \logl^{-A-3}
\end{equation}
under either of the following sets of conditions.
\begin{enumerate}[(a)]
\item \label{type2sumreq} $ \| a \|_2 \| b \|_2 \ll N^{1/2} \logl^2$, $N^{1/2} \leq K \leq N^{3/4}$;
\item \label{type1sumreq} $|a_m| \leq 1$, $b_k =1$ for $k \in I_m \subset [K, 2K)$, $b_k = 0$ otherwise, $M \leq N^{1/4}$.
\end{enumerate}

We use Dirichlet characters to detect the congruence relation in \eqref{afterhdid} and we require the estimate
\[ \sum_{q \leq \min (N^{1/4}, r) N^{-\varepsilon/2}} \frac{\eta_q}{\varphi(q)} \sum_{\chi \bmod{q}} \overline{\chi} (a_q) \mathop{\sum_{M \leq m < 2M} \sum_{K \leq k < 2K}}_{N \leq mk < N'} a_m b_k \chi(mk) e (\gamma mk)  \ll N \logl^{-A-4} . \]
It suffices to show that
\begin{equation} \label{afterhdid2}
S: = \sum_{Q \leq q < 2Q} \sum_{\chi \bmod{q}} \left| \mathop{\sum_{M \leq m < 2M} \sum_{K \leq k < 2K}}_{N \leq mk < N'} a_m b_k \chi(mk) e(\gamma mk) \right| \ll QN \logl^{-A-6}
\end{equation}
for $Q \leq \min (N^{1/4} , r ) N^{-\varepsilon/2}$. \newline

In case \eqref{type2sumreq}, we apply Lemma~\ref{doubleaddmultsum}, which gives
\begin{equation*}
\begin{split}
S & \ll N^{1/2+\varepsilon/6} \left( Q^2 M^{1/2} + \frac{Q^{3/2} N^{1/2}}{r^{1/2}} + Q^{3/2} K^{1/2} + Q^{3/2} r^{1/2} \right) \\
& \ll N^{3/4+\varepsilon/6} Q^2 + \frac{N^{1+\varepsilon/6} Q^{3/2}}{r^{1/2}} + Q^{3/2} N^{7/8+\varepsilon/6} .
\end{split}
\end{equation*}
Each one of these three terms is $\ll QN \logl^{-A-6}$ as
\[ N^{3/4+\varepsilon/6} Q^2 ( QN \logl^{-A-6} )^{-1} \ll QN^{-1/4+\varepsilon/5} \ll 1, \]
\[ N^{1+\varepsilon/6} Q^{3/2} r^{-1/2} (QN \logl^{-A-6})^{-1} \ll Q^{1/2} N^{\varepsilon/4} r^{-1/2} \ll 1 , \]
since $Q \leq r N^{-\varepsilon/2}$, and
\[ N^{7/8+\varepsilon/6} Q^{3/2} (QN \logl^{-A-6})^{-1} \ll N^{-1/8+\varepsilon/5} Q^{1/2} \ll 1 . \] 

In case \eqref{type1sumreq}, we use Lemma~\ref{Ssumestnew}.  Suppose that $K < N^{1-\varepsilon/4}$; \eqref{doubleaddmultsum2} of Lemma~\ref{Ssumestnew} gives
\[ S \ll Q^{3/2} N^{\varepsilon/6} \left( \frac{N}{r} + QM + \frac{K}{Q} + r \right) . \]
Each of the above four terms is $\ll QN \logl^{-A-6}$, since
\[ \frac{Q^{3/2}N^{1+\varepsilon/6}}{r} (QN \logl^{-A-6})^{-1} \ll Q^{1/2} r^{-1} N^{\varepsilon/5} \ll 1, \]
\[ Q^{5/2} N^{\varepsilon/6} M (QN \logl^{-A-6})^{-1} \ll Q^{3/2} N^{-3/4+\varepsilon/5} \ll 1, \]
\[ Q^{1/2} N^{\varepsilon/6} K (QN \logl^{-A-6})^{-1} \ll K N^{-1+\varepsilon/4} \ll 1\]
and
\[ Q^{3/2} N^{\varepsilon/6} r (QN \logl^{-A-6})^{-1} \ll Q^{1/2} N^{-1/4+\varepsilon/5} \ll 1. \]
Now suppose that $K \geq N^{1-\varepsilon/4}$.  Then
\[ 4MQ \ll QN^{\varepsilon/4}, \; \mbox{thus} \; 4MQ < r \]
and
\[ 4MQr \left| \gamma - \frac{u}{r} \right| \ll MQ N^{-3/4}, \; \mbox{hence} \; 4MQr \left| \gamma - \frac{u}{r} \right| \leq \frac{1}{2} .\]
So \eqref{doubleaddmultsum3} of Lemma~\ref{Ssumestnew} gives comfortably:
\[ S \ll N^{\varepsilon} Q^{3/2} r \ll QN \logl^{-A-6} , \]
completing the proof.
\end{proof}

\begin{proof}[Proof of Theorem~\ref{beabardavhal}]
We first show that the contribution to the sum in \eqref{beabardavhaleq} from $q \leq \logl^{A+1}$ is
\[ \ll N^2 \logl^{-A} \ll NR. \]
Since, for some $Q \leq \logl^{A+1}$,
\[ \sum_{q \leq \logl^{A+1}} \sum_{\substack{a=1 \\ (a,q)=1}}^q E^2 \ll N \sum_{q \leq \logl^{A+1}} \frac{1}{\varphi(q)} \sum_{\substack{a=1 \\ (a,q)=1}}^q E(N, N', \gamma, q, a) \ll \frac{N \logl}{Q} \sum_{Q \leq q < 2Q} \max_{(a,q)=1} E(N, N', \gamma, q, a) , \]
it suffices to show for this $Q$ that
\begin{equation} \label{smallqest}
\sum_{Q \leq q < 2Q} \max_{(a,q)=1} E(N,N', \gamma, q, a) \ll QN \logl^{-A-1} .
\end{equation}
We may suppose that $A$ is large.  Arguing as in the proof of Theorem~\ref{beabomvino}, we need only show that \eqref{afterhdid2} follows from either \eqref{type2sumreq} or \eqref{type1sumreq}.  By Dirichlet's theorem, there is a rational approximation $b/r$ to $\gamma$ satisfying \eqref{gammaratapprox}.  For any $\eta > 0$,
\[ N^{-3/4} \geq \| \gamma r \| \gg \exp (- r^{\eta}) , \]
hence $r \gg \logl^{5A}$.  Now we apply Lemma~\ref{doubleaddmultsum} to prove the desired bound under \eqref{type2sumreq}.  Since $D \leq Q \leq \logl^{A+1}$, the term
\[ \| a \|_2 \| b \|_2 \logl^2 D^{1/2} Q^{3/2} H^{1/2} N^{1/2} r^{-1/2} \]
presents no difficulty; the other terms are clearly all small enough.  For the bound under \eqref{type1sumreq}, a similar remark applies to Lemma~\ref{Ssumestnew} and the terms
\[ Q^{3/2} \logl DNHr^{-1} \]
if $K < N^{1-\varepsilon/4}$ and
\[ \logl D Q^{3/2} r \]
if $K \geq N^{1-\varepsilon/4}$.  This establishes \eqref{smallqest}. \newline

It remains to examine the contribution to the sum in \eqref{beabardavhaleq} from $q \in [Q, 2Q)$ with $\logl^{A+1} \leq Q \leq R$.  We have
\begin{equation*}
\begin{split}
 \sum_{Q \leq q < 2Q} \sum_{\substack{a=1 \\ (a,q)=1}}^q E(N, N', \gamma, q, a)^2 \ll \sum_q \sum_a \sup_I & \left| \sum_{\substack{N < n \leq N' \\ \{ \gamma n \} \in I \\ n \equiv a \bmod{q}}} \Lambda(n) - |I| \sum_{\substack{N < n \leq N' \\ n \equiv a \bmod{q}}} \Lambda(n) \right|^2 \\
 &+ \sum_q \sum_a \left( \sum_{\substack{N < n \leq N' \\ n \equiv a \bmod{q}}} \Lambda(n) - \frac{N'-N}{\varphi(q)} \right)^2 = T_1(Q) + T_2(Q),
 \end{split}
 \end{equation*}
say.  Since $T_2(Q)$ is covered by a slight variant of the discussion in \cite[Chapter 29]{HD}, we focus our attention on $T_1(Q)$.  By Lemma~\ref{Baklem},
\begin{equation*}
\begin{split}
T_1(Q) & \ll \sum_{Q \leq q < 2Q} \sum_{\substack{a=1 \\ (a,q)=1}}^q \logl^{-2A} \left( \sum_{\substack{N < n \leq N' \\ n\equiv a \bmod{q}}} \Lambda(n) \right)^2 +\sum_{Q \leq q < 2Q} \sum_{\substack{a=1 \\ (a,q)=1}}^q \left( \sum_{h \leq \logl^A} \frac{1}{h} \left| \sum_{\substack{N < n \leq N' \\ n \equiv a \bmod{q}}} \Lambda(n) e ( \gamma n h) \right| \right)^2 \\
& = T_3(Q) + T_4(Q),
\end{split}
\end{equation*}
say.  The Brun-Titchmarsh Theorem gives a satisfactory bound for $T_3(Q)$.  Applying Cauchy's inequality to $T_4(Q)$, we get
\begin{equation*}
\begin{split}
T_4(Q) & \leq \left( \sum_{h \leq \logl^A} \frac{1}{h} \right) \sum_{h \leq \logl^A} \frac{1}{h} \sum_{Q \leq q < 2Q} \sum_{\substack{a=1 \\ (a,q)=1}}^q \left| \sum_{\substack{N < n \leq N' \\ n \equiv a \bmod{q}}} \Lambda(n) e (\gamma n h ) \right|^2 \\
& \ll (\log \logl)^2 \sum_{Q \leq q < 2Q} \frac{1}{\varphi(q)} \sum_{\chi \bmod{q}} \left| \sum_{N < n \leq N'} \Lambda(n) \chi(n) e(\gamma n h) \right|^2,
\end{split}
\end{equation*}
for some $h \leq \logl^A$.  From this point, we can conclude the proof by following, with slight changes, the argument in \cite[pp. 170-171]{HD}.
\end{proof}

\section{Proof of Theorems \ref{bea1} and \ref{bea2}}

\begin{proof}[Proof of Theorem~\ref{bea1}]
Let $\gamma = \alpha^{-1}$ and $N \geq C_1(\alpha, t)$, $0 < \varepsilon < C_2(\alpha, t)$.  By Dirichlet's theorem, there is a reduced fraction $b/r$ satisfying \eqref{gammaratapprox}.  Our hypothesis on $\alpha$ implies that
\[ N^{-3/4} \geq \| \gamma r \| \gg r^{-3} , \; r \gg N^{1/4} . \]
Let $h_1''$, $\cdots$, $h_l''$ be the first $l$ primes in $(l, \infty)$.  Any translate
\[ \mathcal{H} = \{ h_1' , \cdots, h_k' \} + h , \; h \in \natn \]
with $\{ h_1' , \cdots, h_k' \}  \subset \{ h_1'', \cdots, \; h_l'' \}$, is an admissible set.  Using \eqref{Baklem1} of Lemma~\ref{Baklem}, we choose $h_1'$, $\cdots$, $h_k'$ so that
\begin{equation} \label{klowerbound}
k \geq \varepsilon \gamma l
\end{equation}
and for some real $\eta$,
\[ - \gamma h_m' \in ( \eta, \eta + \varepsilon \gamma ) \pmod{1} \]
for every $m = 1, \cdots , k$.  Now choose $h \in \natn$, $h \ll_{\gamma} 1$ so that
\[ h \gamma \in ( \eta - \varepsilon \gamma, \eta ) \pmod{1} . \]
Thus, writing $h_m = h_m'+h$, we have
\[ - \gamma h_m = - \gamma h_m' - \gamma h \in (0 , 2 \varepsilon \gamma) \pmod{1}  . \]

We apply Theorem~\ref{gentheo} to the set
\[ \mathcal{A} = \{ n \in [N, 2N) : \gamma m \in I \pmod{1} \} \]
where $I = ( \gamma \beta - \gamma, \gamma \beta )$, taking $q_0 = q_1 = 1$, $s=1$, $\varrho(n) = X ( \prip; n)$, $\theta = 1/4 - \varepsilon$, $b = 1-2 \varepsilon$,
\[ Y = \gamma N, \; Y_{1,m} = l_m \int\limits_N^{2N} \frac{1}{\log t} \dif t = \frac{l_m Y}{\logl \gamma} \left( 1 + o(1) \right) . \]
Here $J_m$, $l_m$ are the interval $J$ and its length $l$ in Lemma~\ref{intervallem} (with $\varepsilon \gamma$ in place of $\varepsilon$), so that
\[ \gamma > l_m > \gamma (1-2 \varepsilon) . \]
Since \eqref{regcond1} can be proved in a similar (but simpler) fashion to \eqref{regcond3}, we only show that \eqref{regcond3} holds.  We can rewrite this in the form
\begin{equation} \label{bomvinoalt}
\sum_{q \leq x^{1/4-\varepsilon}} \mu^2(q) \tau_{3k}(q) \left| \sum_{\substack{ N+h_m \leq p < 2N \\ p \equiv a_q \bmod{q} \\ \gamma p \in J_m \bmod{1}}} 1 - \frac{l_m}{\varphi(q)} \int\limits_N^{2N} \frac{\dif t}{\log t} \right| \ll N \logl^{-k-\varepsilon} .
\end{equation}
The function $E(N, N' , \gamma, q, a)$ appearing in Theorem~\ref{beabomvino} is not quite in the form that we need.  However, discarding prime powers and using partial summation in the standard way, we readily deduce a variant of \eqref{bomvinoalt} from Theorem~\ref{beabomvino}, in which $N \logl^{-A}$ appears in place of $N \logl^{-k-\varepsilon}$, and the weight $\mu^2(q) \tau_{3k} (q)$ is absent.  We then obtain \eqref{bomvinoalt} by using Cauchy's inequality; see \cite[(5.20)]{May} for a very similar computation. \newline

We are now in a position to use Theorem~\ref{gentheo}, obtaining a set of $\mathcal{S}$ of $t$ primes in $\mathcal{A} \cap [N, 2N)$, which of course have the form $[\alpha n + \beta]$, with
\[ D(\mathcal{S}) \leq h_k - h_1 \leq h_l'' \]
provided that
\begin{equation} \label{Mkreq}
M_k > \frac{2t-2}{(1-2\varepsilon)(1/4-\varepsilon)} .
\end{equation}
We take $l$ to be the least integer with
\[ \log (\varepsilon \gamma l) \geq \frac{2t-2}{(1-2\varepsilon)(1/4-\varepsilon)} + C \]
for a suitable absolute constant $C$, so that \eqref{Mkreq} follows from \eqref{klowerbound} and \eqref{mklowerbound}.  Therefore,
\[ \gamma l \ll \exp (8t) , \; l \ll \alpha \exp (8t) , \; D( \mathcal{S} ) \ll l \log l \ll \alpha (t + \log \alpha ) \exp (8t) , \]
completing the proof.
\end{proof}

In the proof of Theorem~\ref{bea2}, we shall need the following.

\begin{lemma} \label{intbound}
Let $D$ be as in Lemma~\ref{buchstabdecomp} and let $\omega_0(t)$ denote Buchstab's function.
\begin{enumerate}[(i)]
\item \label{intbound1} The points of $D$ lie in two triangles $A_1$, $A_2$, where $A_1$ has vertices 
\[ \left( \frac{5}{21}, \; \frac{5}{21} \right), \; \left( \frac{2}{7} , \frac{3}{14} \right) , \; \left( \frac{2}{7} , \frac{2}{7} \right) \]
and $A_2$ has vertices
\[ \left( \frac{1}{2}, \; \frac{3}{14} \right), \; \left( \frac{3}{7} , \frac{2}{7} \right) , \; \left( \frac{1}{2} , \frac{1}{4} \right) . \]
\item \label{intbound2} For $j=1,2$, let
\[ I_j = \int\limits_{A_j} \frac{1}{\alpha_1 \alpha_2^2} \omega_0 \left( \frac{1-\alpha_1 - \alpha_2}{\alpha_2} \right) \dif \alpha_1 \dif \alpha_2 . \]
Then $I_1 < 0.03925889$ and $I_2 < 0.0566295$.
\end{enumerate}
\end{lemma}

\begin{proof} Let $(\alpha_1, \alpha_2) \in D$.  If $\alpha_1 + \alpha_2 > 5/7$, then we have
\[ \alpha_1 + \alpha_2 > \frac{5}{7}, \; \alpha_1 + 2 \alpha_2 \leq 1, \; \alpha_1 \leq \frac{1}{2} . \]
This defines a triangle which is easily verified to be $A_2$.  If $\alpha_1 + \alpha_2 \leq 5/7$, then as $\mathbold{\alpha}_2$ is bad, we have in turn
\[ \alpha_1 + \alpha_2 < \frac{4}{7} , \; \alpha_1 < \frac{3}{7} , \; \alpha_1 < \frac{2}{7} . \]
Altogether, we have
\[ \alpha_1 + 2 \alpha_2 > \frac{5}{7} , \; \alpha_1 < \frac{2}{7} , \; \alpha_2 < \alpha_1 . \]
This defines a triangle which we can verify to be $A_1$.  This proves \eqref{intbound1}.  Now \eqref{intbound2} requires a computer calculation, which was kindly carried out by Andreas Weingartner.
\end{proof}

\begin{proof} [Proof of Theorem~\ref{bea2}]
With a different value of $l$, we choose $h_1''$, $\cdots$, $h_l''$ and $h_1$, $\cdots$, $h_k$ exactly as in the proof of Theorem~\ref{bea1}.  In applying Theorem~\ref{gentheo}, we also take $I$, $\mathcal{A}$, $q_0$, $q_1$, $Y$, $J_m$, $l_m$ as in that proof, but now $\theta=2/7-\varepsilon$, $s=5$, $a=3$; the functions $\varrho_1(n)$, $\cdots$, $\varrho_5(n)$ are given in Lemma~\ref{buchstabdecomp}. \newline

There is little difficulty in verifying \eqref{regcond1} by a similar but simpler version of the proof of \eqref{regcond3}.  So we concentrate on \eqref{regcond3}.  We recall that this can be rewritten as
\begin{equation} \label{bomvinoalt2}
\sum_{q \leq x^{\theta}} \mu^2 (q) \tau_{3k} (q) \left| \sum_{\substack{ n \equiv a_q \bmod{q} \\ \gamma n \in J_m \bmod{1} \\ N + h_m \leq n < 2N}} \varrho_g(n) - \frac{Y_{g,m}}{\varphi(q)} \right| \ll N \logl^{-k-\varepsilon} .
\end{equation}
We define $Y_{g,m}$ by
\[ Y_{g,m} = l_m \sum_{N \leq n < 2N} \varrho_g(n) . \]
It is well known that
\begin{equation} \label{Ygmapprox}
Y_{g,m} = \frac{l_m c_g N}{\logl} \left( 1 + o(1) \right) ,
\end{equation}
where $c_g$ is given by a multiple integral.  In fact, we have
\[ c_1 + c_2 + c_3 - c_4 - c_5 = 1 - \int\limits_{\mathbold{\alpha}_2 \in D} \frac{1}{\alpha_1 \alpha_2^2} \omega_0 \left( \frac{1-\alpha_1 - \alpha_2}{\alpha_2} \right) \dif \alpha_1 \dif \alpha_2 . \]
Similar calculations are found in \cite[Chapter 1]{GH}. \newline

Fix $m$ and $g$.  By analogy with the proof of Theorem~\ref{beabomvino}, we can obtain \eqref{regcond3} by showing
\begin{equation} \label{regcond3req1}
\sum_{q \leq N^{2/7-\varepsilon}} \left| \sum_{\substack{N \leq n < N' \\ n \equiv a_q \bmod{q}}} \varrho_g(n) - \frac{1}{\varphi(q)} \sum_{N \leq n < N'} \varrho_g(n) \right| \ll N \logl^{-A}
\end{equation}
for every $A>0$ and
\begin{equation} \label{regcong3req2}
\sum_{q \leq N^{2/7-\varepsilon}} \left| \sum_{\substack{N \leq n < N' \\ n \equiv a_q \bmod{q}}} \varrho_g(n) e ( \gamma nh) \right| \ll N \logl^{-A} 
\end{equation}
for $1 \leq h \leq \logl^{A+1}$ and for every $A > 0$.  Again adapting the argument of Theorem~\ref{beabomvino}, we see that \eqref{regcong3req2} is a consequence of Lemma~\ref{primechardetect}. \newline

For \eqref{regcond3req1}, it suffices to show, recalling Lemma~\ref{splitdiag}, that for arbitrary $\eta_{\chi} \ll 1$ and $Q \leq N^{2/7-\varepsilon}$,
\begin{equation} \label{regcond3req1redprim}
\sum_{Q \leq q < 2Q}  \ \sideset{}{^{\star}} \sum_{\chi \bmod{q}} \eta_{\chi} \sum_{N \leq n < N'} \varrho_g(n) \chi(n) \ll QN\logl^{-A}
\end{equation}
for every $A > 0$.  This can be readily deduced from the Siegel-Walfisz theorem for $Q \leq \logl^{2A}$, so we assume that $Q > \logl^{2A}$. \newline

We apply Lemma~\ref{BaWelem} with
\[ W (n) = \sum_{Q \leq q < 2Q} \ \sideset{}{^{\star}} \sum_{\chi \bmod{q}} \eta_{\chi} \chi(n) \]
if $N \leq n < N'$ and $W(n)=0$ otherwise. \newline

For example, when $g=3$, the left-hand side of \eqref{regcond3req1redprim} is
\[ \sum_{\substack{N \leq p_1 p_2 p_3 n_4 < N' \\ (n_4, P((2N)^{1/7}) = 1 \\ \mathbold{\alpha}_3 \in E_3 \\ (\alpha_1, \alpha_2) \in E_2 \setminus D}} W (p_1 p_2 p_3 n_4) = \sum_{\substack{ \mathbold{\alpha}_3 \in E_3 \\ ( \alpha_1 , \alpha_2) \in E_2 \setminus D}} S^* (p_1 p_2 p_3, (2N)^{1/7} ) . \]
We shall show that \eqref{type1sumest} and \eqref{type2sumest} hold with $Y = QN \logl^{-A-3}$, $c = 4/7$ and $d=1/7$. (We could reduce the constraints on $c$ and $d$, but that would not be useful in the present context.)  Once we have done this, we can follow the proof of Lemma~\ref{primechardetect} to prove \eqref{regcond3req1redprim}. \newline

To prove \eqref{type1sumest}, we use the Polya-Vinogradov bound for character sums to obtain
\begin{equation*}
\begin{split}
 \sum_{m \leq 2N^{4/7}} \sum_k W(nk) & = \sum_{m \leq 2N^{4/7}} a_m \sum_{Q \leq q < 2Q} \sideset{}{^{\star}} \sum_{\substack{\chi \bmod{q} \\ N \leq mk < N'}} \eta_{\chi} \chi(mk) \\
 & \ll \logl \sum_{m \leq 2N^{4/7}} \sum_{Q \leq q < 2Q} q^{1/2} \ll \logl Q^{3/2} N^{4/7-\varepsilon} \ll QN \logl^{-A-3} .
 \end{split}
 \end{equation*}

Now to prove \eqref{type2sumest}, we note that by the method of \cite[Section 3.2]{GH} mentioned earlier, it suffices to show that
\[ \sum_{M \leq m < 2M} a_m \sum_{K \leq k < 2K} b_k W(mk) \ll QN \logl^{-A} \]
whenever $|a_m| \leq 1$ and $|b_k| \leq \tau(k)$, $N^{4/7} \ll M \ll N^{5/7}$, $MK \asymp N$.  That is, it suffices to show that
\begin{equation} \label{beforelargesieve}
\sum_{Q \leq q < 2Q} \ \sideset{}{^{\star}} \sum_{\chi \bmod{q}} \left| \sum_{M \leq m < 2M} a_m \chi(m) \right| \left| \sum_{K \leq k < 2K} b_k \chi(k) \right| \ll Q N \logl^{-A} .
\end{equation}
Following the proof of (6) in \cite[Chapter 28]{HD}, the left-hand side of \eqref{beforelargesieve} is
\[ \ll \logl ( M+Q^2)^{1/2} (K+Q^2)^{1/2} \| a \|_2 \| b \|_2 \ll \logl^3 \left( N^{1/2} + M^{1/2} Q + Q^2 \right) N^{1/2} \ll QN \logl^{-A} , \]
since $\logl^3 Q^{-1} N \ll \logl^{3-A} N$, $\logl^3 M^{1/2} N^{1/2} \ll \logl^3 N^{6/7} \ll N \logl^{-A}$ and $\logl^3 Q N^{1/2} \ll \logl^3 N^{11/14} \ll N \logl^{-A}$.  This proves \eqref{regcond3} with the present choice of $\mathcal{A}$, $Y_{g,m}$, etc. \newline

Applying Theorem~\ref{gentheo}, we find that there is a set $\mathcal{S}$ of $t$ primes in $\mathcal{A}$ (and thus of the form $[ \alpha m + \beta ]$) having diameter
\[ D(\mathcal{S}) \leq h_k - h_1 \ll l \log l \]
provided that
\[ M_k > \frac{2t-2}{b (2/7- \varepsilon)} . \]
Here $b$ must have the property
\[ b_{1,m} + b_{2,m} + b_{3,m} - b_{4,m} - b_{5,m}  \geq b > 0 ; \]
that is, 
\[ l_m (c_1 + c_2 + c_3 - c_4 - c_5 ) \geq b \gamma > 0 . \]
We can choose
\[ b = (1-2 \varepsilon) \left( 1- \int\limits_{\mathbold{\alpha}_2 \in D} \frac{1}{\alpha_1 \alpha_2^2} \omega_0 \left( \frac{1-\alpha_1 - \alpha_2}{\alpha_2} \right) \dif \alpha_1 \dif \alpha_2 \right) . \]
Using Lemma~\ref{intbound}, we see that
\[ b > 0.90411. \]
Now we proceed just as the proof of Theorem~\ref{bea1}.  We may choose any $l$ for which
\[ \log ( \varepsilon \gamma l ) \geq \frac{2t-2}{0.90411 (2/7-\varepsilon)} + C \]
for a suitable constant $C$, and now it is a simple matter to deduce that
\[ D( \mathcal{S} ) < C_4 \alpha ( \log \alpha + t) \exp ( 7.743 t) ,\]
where $C_4$ is an absolute constant.
\end{proof}

\noindent{\bf Acknowledgments.}  This work was done while L. Z. held a visiting position at the Department of Mathematics of Brigham Young University (BYU).    He wishes to thank the warm hospitality of BYU during his thoroughly enjoyable stay in Provo.

\bibliography{biblio}
\bibliographystyle{amsxport}

\vspace*{.5cm}

\noindent\begin{tabular}{p{8cm}p{8cm}}
Roger C. Baker & Liangyi Zhao \\
Department of Mathematics & School of Mathematics and Statistics \\
Brigham Young University & University of New South Wales\\
Provo, UT 84602, U. S. A. & Sydney, NSW 2052 Australia \\
Email: {\tt baker@math.byu.edu} & Email: {\tt l.zhao@unsw.edu.au} \\
\end{tabular}
\end{document}